\documentclass[nointlimits,11pt,oneside]{amsart}
\usepackage[utf8]{inputenc}
\usepackage{amsmath,amsfonts,amsthm,amssymb,cases,enumerate}
\usepackage{xcolor}
\usepackage[colorlinks, linkcolor=teal, citecolor=magenta]{hyperref}
\usepackage{mathscinet}
\usepackage[square,numbers]{natbib}
\usepackage{microtype}
\usepackage[mathcal]{euscript}
\usepackage[%
	a4paper,
	left=1.5cm,
	right=1.5cm,
	top=3cm,
	bottom=3cm] {geometry}

\theoremstyle{plain}
\newtheorem{theorem}{Theorem}[section]
\newtheorem{lemma}[theorem]{Lemma}
\theoremstyle{definition}
\newtheorem{remark}[theorem]{Remark}
\newtheorem{definition}[theorem]{Definition}
\numberwithin{equation}{section}

\newcommand{\reftext}[1]{#1}
\newcommand{\Cases}[1]{\begin{cases}#1\end{cases}}
\newcommand{\atopfrac}[2]{\substack{#1 \\ #2}}

\makeatletter
\def\paragraph{\bigskip\@startsection{paragraph}{4}%
  \z@\z@{-\fontdimen2\font}%
  {\normalfont\bfseries}}
\makeatother

\begin{document}

\title{Boundedness of classical operators on rearrangement-invariant spaces}

\author{David E. Edmunds, Zden\v ek Mihula, V\'{\i}t Musil and Lubo\v s Pick}

\address{David E. Edmunds, Department of Mathematics, University of Sussex, Falmer, Brighton, BN1 9QH, UK}
\email{davideedmunds@aol.com}
\urladdr{0000-0003-2394-9385}

\address{Zden\v ek Mihula, Department of Mathematical Analysis, Faculty of Mathematics and
Physics, Charles University, Sokolovsk\'a~83,
186~75 Praha~8, Czech Republic}
\email{mihulaz@karlin.mff.cuni.cz}
\urladdr{0000-0001-6962-7635}

\address{V\'{\i}t Musil, Department of Mathematical Analysis, Faculty of Mathematics and
Physics, Charles University, Sokolovsk\'a~83,
186~75 Praha~8, Czech Republic}
\email{musil@karlin.mff.cuni.cz}
\urladdr{0000-0001-6083-227X}

\address{Lubo\v s Pick, Department of Mathematical Analysis, Faculty of Mathematics and
Physics, Charles University, Sokolovsk\'a~83,
186~75 Praha~8, Czech Republic}
\email{pick@karlin.mff.cuni.cz}
\urladdr{0000-0002-3584-1454}


\date{31/10/2019}

\begin{abstract}
We study the behaviour on rearrangement-invariant (r.i.) spaces of such
classical operators of interest in harmonic analysis as the
Hardy-Littlewood maximal operator (including the fractional version),
the Hilbert and Stieltjes transforms, and the Riesz potential. The focus
is on sharpness questions, and we present characterisations of the
optimal domain (or range) partner spaces when the range (domain) is
fixed. When an~r.i.~partner space exists at all, a
complete characterisation of the situation is given. We illustrate the
results with a variety of examples of sharp particular results involving
customary function spaces.
\end{abstract}

\bibliographystyle{abbrvnat}

\maketitle

\section*{How to cite this paper}
\noindent
This paper has been accepted for publication in \emph{Journal of Functional
Analysis} and is available on
\begin{center}
	\url{https://doi.org/10.1016/j.jfa.2019.108341}.
\end{center}
Should you wish to cite this paper, the authors would like to cordially ask you
to cite it appropriately.

\section{Introduction} \label{sec1}

\noindent
Given function spaces $X,Y$ and an operator $T$ that maps $X$ boundedly
into $Y$, it is natural to ask whether there is a space bigger than
$X$ that is also mapped boundedly by $T$ into $Y$, or a space smaller
than $Y$ into which $T$ maps $X$ boundedly.

Such questions have been attracting a great deal of attention for many
years, in particular in connection with embeddings of Sobolev spaces,
see, for
example~\cite{BMR:03,CarSo:97,ClaSo:16a,ClaSo:16,CwPu:98,FiRa:06,Ker:79,MaMi:06,MMP:06,MaMi:10,Tal:94,Tal:16,Tar:98}.
By way of illustration we consider a particularly simple Sobolev
embedding. Let $\Omega $ be a bounded open subset of $\mathbb{R}^{n}$,
let $p\in [1,n)$ and put $p^{\ast }=np/(n-p)$. It is classical that, in
standard notation, the Sobolev space $W_{0}^{1,p}(\Omega )$ is embedded
in $L^{p^{\ast }}(\Omega )$. Can $W_{0}^{1,p}(\Omega )$ be embedded in
a space smaller than $L^{p^{\ast }}(\Omega )$? Is there a space larger
than $W_{0}^{1,p}(\Omega )$ that can be embedded in $L^{p^{\ast }}(
\Omega )?$ To make such questions sensible the class of competing spaces
must be specified. If we restrict ourselves to Lebesgue spaces as
targets and domain spaces that are Sobolev spaces based on Lebesgue
spaces, then the embedding $W_{0}^{1,p}(\Omega )\hookrightarrow $
$L^{p^{\ast }}(\Omega )$ is optimal in the sense that neither the domain
nor the target space can be improved. This leaves open the question of
optimality in classes of spaces wider than those involving the Lebesgue
scale. If the class of admissible target spaces is taken to be that of
rearrangement-invariant (r.i.) spaces, then the optimal range space
turns out to be the Lorentz space $L^{p^{\ast },p}(\Omega )$\textup{;} there is
a similar improvement of the domain space, involving a Sobolev space
based on a Lorentz rather than a Lebesgue space.

The first results in this direction were obtained in~\cite{EKP}
in connection with re\-ar\-range\-ment-invariant quasinorms. Further
extensions concerning r.i.~norms were added later in
several papers, for instance~\cite{T2,T3}. A comprehensive
treatment of optimal Sobolev embeddings on Euclidean domains equipped
with general measures having specific isoperimetric properties was given
in \citep{CPS}.

Embeddings are not the only maps for which such questions are of
interest and importance. The optimality of r.i.~spaces on
which the Laplace transform $\mathcal{L}$ acts boundedly was studied in
a recent paper~\citep{BEP}. A special case of the results obtained is
that if $p\in (1,\infty )$ and $q\in \lbrack 1,\infty ]$, then
$\mathcal{L}$ maps the Lorentz space $L^{p,q}(0,\infty )$ boundedly into
$L^{p^{\prime },q}(0,\infty )$, a fact which we denote by $
\mathcal{L}\colon L^{p,q}(0,\infty )\rightarrow L^{p^{\prime },q}(0,
\infty )$. Moreover, both the domain and target spaces are optimal:
there is no r.i. space smaller than $L^{p^{\prime },q}(0,\infty )$ into
which $\mathcal{L}$ maps $L^{p,q}(0,\infty )$, and there is no
r.i. space larger than $L^{p,q}(0,\infty )$ mapped by $\mathcal{L}$
into $L^{p^{\prime },q}(0,\infty )$. Thus in particular $\mathcal{L}
\colon L^{p}(0,\infty )\rightarrow L^{p^{\prime },p}(0,\infty )$ and the
spaces involved form an optimal pair; if $p>2$, there is no $q$ for
which $\mathcal{L}\colon L^{p}(0,\infty )\rightarrow L^{q}(0,\infty )$.

In the present paper we discuss such problems for classical operators
of great interest in analysis and its applications, namely the Hilbert
and Stieltjes transforms, the Riesz potential and various versions of
the maximal operator. The action of these operators on specific classes
of function spaces has been extensively studied over several
decades. Classical results are available for example in connection with
familiar function spaces. The 1970s experienced a real boom of this
theory involving weighted Lebesgue spaces and fundamental papers were
written (\cite{Muc:72,Sa:82} for the Hardy--Littlewood maximal
operator,~\cite{MW:71} for singular and fractional
integrals,~\cite{CF:74,HMW,MW:76} for the Hilbert transform).
Later it became apparent that Lebesgue spaces are not sufficient for
describing all the important situations and other function spaces were
investigated. Classical Lorentz spaces which originated in the 1950s and
have been occurring occasionally later (see~\cite{Bag:83,Boy:67})
became extremely fashionable in the 1990s when the fundamental
papers~\cite{AM:90,Sa:90} appeared. Various important and deep
results were obtained, see for
example~\cite{ACS:12,CO:15,CS:93,CS:97,CE:97}. Orlicz spaces which
generalize Lebesgue's scale in a direction essentially different from
Lorentz spaces, received much attention too, see for
instance~\cite{BK:84,BP:87,Cia:97,Cia:99,CM:19,Gal:88,Mus:19}. The
results naturally found their way into important monographs that are
considered classic these days,
see~\cite{CMP:11,DHHR:11,GR:85,Mabook,Ruz:00,Ste:70,Ste:93,SW:71}.
Let us point out that, in particular, in the
monograph~\cite{Mabook}, among plenty of other fundamental results,
the significance of the connection between embeddings and integral
operators is explained in great detail.

On the other hand, surprisingly little attention has been paid to the
\textit{sharpness} of the results, perhaps with an exception of results
in different direction on optimality obtained
e.g.~in~\cite{DS:07,ST:16} and the references therein, where
operators related to the Hardy averaging operator are studied, see
also~\cite{CR:02}. Optimal range spaces for Calder\'on operators are studied in the recent paper~\cite{STZ:19}.

In this paper we study the behaviour of classical operators on
r.i.~spaces, a class of function spaces that includes
for example all Lebesgue, Lorentz, Orlicz, Lorentz-Zygmund spaces and
more. Our focus is mainly on the optimality of function spaces.

We use the Hardy-Littlewood maximal operator $M$ to illustrate the
results obtained and serve as an appetiser for the forthcoming
attractions. Let $X$ be an~r.i.~space
over $\mathbb{R}^{n}$ with associate space $X^{\prime }$\textup{;} denote by
$X^{\prime }(0,\infty )$ the representation space of $X^{\prime }$ and
suppose that the function $\psi $ given by $\psi (t)=\chi _{(0,1)}(t)
\log (1/t)$ belongs to $X^{\prime }(0,\infty )$. Let $Y^{\prime }$ be
the set of all $f$ such that
\begin{eqnarray*}
\varrho (f)=\left \Vert \int \nolimits _{t}^{\infty }f^{\ast }(s)s^{-1}
\,\mathrm{d}s\right \Vert _{X^{\prime }(0,\infty )}<\infty .
\end{eqnarray*}
Endowed with the norm $\varrho $, $Y^{\prime }$ is an r.i. space with
associate space $Y$ that not only has the property that $M\colon X
\rightarrow Y$, but is also the optimal range space corresponding to
$X$. If $\psi \notin X^{\prime }(0,\infty )$, there is no r.i. space
$Z$ over $\mathbb{R}^{n}$ such that $M\colon X\rightarrow Z$.

The situation turns out to be considerably more complicated in the case
of the fractional maximal operator, another classical operator of
harmonic analysis. The reason is that the appropriate analogue of the
Riesz--Wiener--Herz inequality for the fractional maximal operator leads
to an~inevitable involvement of a~supremum type operator, rather than
just an~integral mean. Supremum operators are not linear and in general
are less manageable than their integral companions. However, using a
fine analysis combining known and new techniques and various delicate
estimates we are able to characterize the optimal range space for this
operator as well. Since the general resulting condition is however
naturally not so simple as in the case of the operator $M$, we include
another, simpler characterization, available under a rather mild extra
assumption. We also include an~interesting and perhaps somewhat
surprising result describing a~vital link between optimality properties
of a space and boundedness of a supremum operator on its associate space
that leads to a~self-explanatory characterization of the above-mentioned
extra condition. This part of the paper is one of the most innovative
ones.

We finally consider two other classical operators of harmonic analysis,
namely the Hilbert transform and the Riesz potential. The importance of
these operators is very well known, and their properties have been
deeply studied. Our contribution is the characterization of the
optimality of the spaces involved. In case of the Hilbert transform we
use the Stieltjes transform as the appropriate tool and obtain
characterizations for it as well.

For each of the operators considered, we are also able to nail down the
optimal domain partner when the range space is fixed, this task being
in general slightly simpler than the converse one. To establish all
this, a~combination of new techniques developed here with those
from~\citep{CPS,EKP} and~\citep{T2} is used.

We illustrate the results obtained with variety of nontrivial examples.
For instance, we recover the well-known fact that
\begin{eqnarray*}
M\colon L(\log L)^{\alpha }(Q)\to L(\log L)^{\alpha -1}(Q)
\end{eqnarray*}
when $\alpha \geq 1$, $Q\subset \mathbb{R}^{n}$ is a cube of finite
measure and $L(\log L)^{\alpha }(Q)$ is the classical Zygmund class
defined as the collection of all measurable functions $g$ on $Q$
satisfying $\int _{Q}|g(x)|(\log (1+|g(x)|)^{\alpha }\mathrm{d}x<\infty $, but we add the information that the range space cannot be
improved in any way when the competing spaces are rearrangement
invariant. Similar examples are even more interesting when the functions
act on a~set of unbounded measure, say, $\mathbb{R}^{n}$. We will for
example prove that if $X$ is the space equipped with the norm
$\|f\|_{X}=\int _{0}^{\infty }f^{*}(t)w(t)\,dt$, where
\begin{eqnarray*}
w(t)=(1-\log t)^{\alpha _{0}}\chi _{(0,1)}+(1+\log t)^{\alpha _{\infty }}
\chi _{[1,\infty )}
\end{eqnarray*}
and $\alpha _{0}\geq 1$ and $\alpha _{\infty }\in [-1,0]$, then the
optimal (smallest possible) r.i.~range space $Y$ such
that
\begin{eqnarray*}
M\colon X(\mathbb{R}^{n})\to Y(\mathbb{R}^{n})
\end{eqnarray*}
is the space whose associate space has norm
\begin{eqnarray*}
\|f\|=\sup _{0<t<\infty }w(t)\sp{-1}\int _{t}^{\infty }f^{*}(s)\,\frac{
\mathrm{d}s}{s},\quad f\in \mathcal{M}_{+}(\mathbb{R}^{n}).
\end{eqnarray*}
Such results have not been available before, and the latter norm cannot
be identified with any customary known one.

We get analogous sets of examples for other operators, too. For example
in the case of the fractional maximal operator we essentially improve
some results from earlier papers such
as~\cite{EO,EOP,EOP-broken,OP}.

\section{Preliminaries} \label{sec2}

\noindent
In this section we collect all the background material that will be used
in the paper. We start with the operation of the nonincreasing
rearrangement of a measurable function.

Throughout this section, let $(R,\mu )$ be a $\sigma $-finite nonatomic
measure space. We set
\begin{eqnarray*}
\mathcal{M}(R,\mu )= \{f: f \mbox{ is a } \mu {-}\mbox{measurable function on }R \mbox{ with values in }[-\infty ,\infty ]\},
\end{eqnarray*}
\begin{eqnarray*}
\mathcal{M}_{0}(R,\mu )= \{f \in \mathcal{M}(R,\mu ): f  \mbox{ is
finite } \mu \mbox{-a.e. on } R\}
\end{eqnarray*}
and
\begin{eqnarray*}
\mathcal{M}_{+}(R,\mu )= \{f \in \mathcal{M}(R,\mu ): f \geq 0\}.
\end{eqnarray*}
The \textit{nonincreasing rearrangement} $f^{*} \colon [0,\infty )
\to [0, \infty ]$ of a function $f\in \mathcal{M}(R,\mu )$ is defined
as
\begin{eqnarray*}
f^{*}(t)=\inf \{\lambda \in (0,\infty ): \mu(\{s\in R: |f(s)|>
\lambda \})\leq t\},\ t\in [0,\infty ).
\end{eqnarray*}
The \textit{maximal nonincreasing rearrangement} $f^{**} \colon (0,
\infty ) \to [0, \infty ]$ of a function $f\in \mathcal{M}(R,\mu )$ is
defined as
\begin{eqnarray*}
f^{**}(t)=\frac{1}{t}\int _{0}^{t} f^{*}(s)\,\mathrm{d}s,\quad t\in (0,\infty ).
\end{eqnarray*}
If $|f|\leq |g|$ $\mu $-a.e. in $R$, then $f^{*}\leq g^{*}$. The
operation $f\mapsto f^{*}$ does not preserve sums or products of
functions, and is known not to be subadditive. The lack of subadditivity
of the operation of taking the nonincreasing rearrangement is, up to
some extent, compensated by the following fact
\citep[Chapter~2,~(3.10)]{BS}: for every $t\in (0,\infty )$ and every
$f,g\in \mathcal{M}(R,\mu )$, we have
%
\begin{eqnarray}
\label{E:subadditivity-of-doublestar}
\int _{0}^{t}(f +g)^{*}(s)\,\mathrm{d}s
\leq \int _{0}^{t}f^{*}(s)\,\mathrm{d}s + \int _{0}^{t}g
^{*}(s)\,\mathrm{d}s.
\end{eqnarray}
This inequality can be also written in the form
%
\begin{eqnarray}
\label{E:subadditivity-of-doublestar-a}
(f+g)^{**}\leq f^{**}+g^{**}.
\end{eqnarray}
A fundamental result in the theory of Banach function spaces is the
\textit{Hardy lemma} \citep[Chapter~2, Proposition~3.6]{BS} which
states that if two nonnegative measurable functions $f,g$ on
$(0,\infty )$ satisfy
\begin{eqnarray*}
\int _{0}^{t}f(s)\,\mathrm{d}s\leq \int _{0}^{t}g(s)\,\mathrm{d}s
\end{eqnarray*}
for all $t\in (0,\infty )$, then, for every nonnegative nonincreasing
function $h$ on $(0,\infty )$, one has
\begin{eqnarray*}
\int _{0}^{\infty }f(s)h(s)\,\mathrm{d}s\leq \int _{0}^{\infty }g(s)h(s)\,\mathrm{d}s.
\end{eqnarray*}

Another important property of rearrangements is the
\textit{Hardy-Littlewood inequality}
\citep[Chapter~2, Theorem~2.2]{BS}, which asserts that, if
$f, g \in \mathcal{M}(R,\mu )$, then
%
\begin{eqnarray}
\label{E:HL}
\int _{R} |fg| \,\mathrm{d}\mu \leq \int _{0}^{\infty } f^{*}(t) g^{*}(t)\,\mathrm{d}t.
\end{eqnarray}

If $(R,\mu )$ and $(S,\nu )$ are two (possibly different)
$\sigma $-finite measure spaces, we say that functions $f\in
\mathcal{M}(R,\mu )$ and $g\in \mathcal{M}(S,\nu )$ are
\textit{equimeasurable}, and write $f\sim g$, if $f^{*}=g^{*}$ on
$(0,\infty )$.

A functional $\varrho \colon \mathcal{M}_{+} (R,\mu ) \to [0,\infty ]$
is called a \textit{Banach function norm} if, for all $f$, $g$ and
$\{f_{j}\}_{j\in \mathbb{N}}$ in $\mathcal{M}_{+}(R,\mu )$, and every
$\lambda \geq 0$, the following properties hold:
\begin{enumerate}[(P1)]%
\item[(P1)]
$\varrho (f)=0$ if and only if $f=0$; $\varrho (\lambda f)= \lambda
\varrho (f)$; $\varrho (f+g)\leq \varrho (f)+ \varrho (g)$ (the
\textit{norm axiom});
\item[(P2)]
$ f \le g$ a.e. implies $\varrho (f)\le \varrho (g)$ (the
\textit{lattice axiom});
\item[(P3)]
$ f_{j} \nearrow f$ a.e. implies $\varrho (f_{j}) \nearrow \varrho (f)$
(the \textit{Fatou axiom});
\item[(P4)]
$\varrho (\chi _{E})<\infty $ for every $E\subset R$ of finite measure
(the \textit{nontriviality axiom});
\item[(P5)]
if $E$ is a subset of $R$ of finite measure, then $\int _{E} f\,{\mathrm{d}}\mu \le C_{E} \varrho (f)$ for some positive constant $C_{E}$, depending on $E$ and $\varrho $ but independent of $f$ (the
\textit{local embedding in $L^{1}$}).
\end{enumerate}
If, in addition, $\varrho $ satisfies
\begin{itemize}%
\item[(P6)] $\varrho (f) = \varrho (g)$ whenever $f^{*} = g^{*}$(the
\textit{rearrangement-invariance axiom}),
\end{itemize}
then we say that $\varrho $ is an
\textit{r.i.~norm}.

If $\varrho $ is an~r.i.~norm, then the collection
\begin{eqnarray*}
X=X({\varrho })=\{f\in \mathcal{M}(R,\mu ): \varrho (|f|)<\infty
\}
\end{eqnarray*}
is called a~\textit{rearrangement-invariant~space} (\textit{r.i.~space} for short), corresponding to the norm
$\varrho $. We shall write $\|f\|_{X}$ instead of $\varrho (|f|)$. Note
that the quantity $\|f\|_{X}$ is defined for every $f\in \mathcal{M}(R,
\mu )$, and
\begin{eqnarray*}
f\in X\quad \Leftrightarrow \quad \|f\|_{X}<\infty .
\end{eqnarray*}

With any r.i.~norm $\varrho $ is associated
another functional, $\varrho '$, defined for $g \in \mathcal{M}_{+}(R,
\mu )$ as
\begin{eqnarray*}
\varrho '(g)=\sup \left \{  \int _{R} fg\,{\mathrm{d}}\mu : f\in \mathcal{M}_{+}(R,\mu ),\ \varrho (f)\leq 1\right \}
.
\end{eqnarray*}
It turns out that $\varrho '$ is also an~r.i.~norm,
which is called the~\textit{associate norm} of $\varrho $. Moreover, for
every r.i.~norm $\varrho $ and every $f\in
\mathcal{M}_{+}(R,\mu )$, we have
(see~\citep[Chapter~1, Theorem~2.9]{BS})
\begin{eqnarray*}
\varrho (f)=\sup \left \{  \int _{R}fg\,{\mathrm{d}}\mu : g\in \mathcal{M}_{+}(R,\mu ),\ \varrho '(g)\leq 1\right \}
.
\end{eqnarray*}
If $\varrho $ is an~r.i.~norm, $X=X({\varrho })$ is
the r.i.~space determined by $\varrho $, and
$\varrho '$ is the associate norm of $\varrho $, then the function space
$X({\varrho '})$ determined by $\varrho '$ is called the
\textit{associate space} of $X$ and is denoted by $X'$. We always have
$(X')'=X$, and we shall write $X''$ instead of $(X')'$. Furthermore, the
\textit{H\"{o}lder inequality}
\begin{eqnarray*}
\int _{R}fg\,{\mathrm{d}}\mu \leq \|f\|_{X}\|g\|_{X'}
\end{eqnarray*}
holds for every $f,g\in \mathcal{M}(R,\mu )$.

An important consequence of the Hardy lemma, which plays a crucial role
in the theory of rearran\-gement-invariant spaces, is the
\textit{Hardy--Littlewood--P\'{o}lya principle}
\citep[Chapter~2, Theorem~4.6]{BS} which asserts that if two functions
$f,g$ satisfy the so-called
\textit{Hardy--Littlewood--P\'{o}lya relation}, defined by
\begin{eqnarray*}
\int _{0}^{t}f^{*}(s){\mathrm{d}}s\leq \int _{0}^{t}g^{*}(s){\mathrm{d}}s, \quad t\in (0,\infty ),
\end{eqnarray*}
and sometimes denoted by $f\prec g$ in the literature, then
$\|f\|_{X}\leq \|g\|_{X}$ provided that the underlying measure space is
resonant. We note that throughout this paper we work solely on nonatomic
measure spaces, which are resonant by~\citep[Chapter~2, Theorem~2.7]{BS}.

For every r.i.~space $X$ over the measure space
$(R,\mu )$ there exists a~unique rearran\-gement-invariant space
$X(0,\mu (R))$ over the interval $(0,\mu (R))$ endowed with the
one-dimensional Lebesgue measure such that $\|f\|_{X}=\|f^{*}\|_{X(0,
\mu(R))}$. This space is called the~\textit{representation space} of
$X$. This follows from the Luxemburg representation theorem
\citep[Chapter~2, Theorem~4.10]{BS}. Throughout this paper, the
representation space of an~r.i.~space $X$ will be
denoted by $X(0,\mu (R))$. It will be useful to notice that when
$R=(0,\infty )$ and $\mu $ is the Lebesgue measure, then every $X$ over
$(R,\mu )$ coincides with its representation space.

If $\varrho $ is an~r.i.~norm and $X=X({\varrho })$
is the r.i.~space determined by $\varrho $, we define
its \textit{fundamental function}, $\varphi _{X}$, for every
$t\in [0,\mu (R))$ by $\varphi _{X}(t)=\varrho (\chi _{E})$, where
$E\subset R$ is such that $\mu (E)=t$. The properties of
r.i.~norms and the fact that the underlying measure space is nonatomic guarantee that the fundamental function
is well defined. Moreover, one has
%
\begin{eqnarray}
\label{E:fundamental-relation}
\varphi _{X}(t)\varphi _{X'}(t)=t, \quad t\in [0,\mu (R)).
\end{eqnarray}

Let $X$ and $Y$ be r.i.~spaces over $(0,\infty )$ and
let $I\colon [0,\infty )\to [0,\infty )$ be a nondecreasing function.
Then
%
\begin{eqnarray}
\label{T:Lenka-unrestricted}
\left \|  \int _{t}^{\infty }\frac{f(s)}{I(s)}\,{\mathrm{d}}s\right \|  _{Y(0,\infty )}
\le C_{1} \|f\|_{X(0,\infty )}
\quad \mbox{for every }f\in \mathcal{M}_{+}(0,\infty )
\end{eqnarray}
holds true with some positive constant $C_{1}$ if and only if
%
\begin{eqnarray}
\label{T:Lenka-nonincreasing}
\left \|  \int _{t}^{\infty }\frac{g(s)}{I(s)}\,{\mathrm{d}}s\right \|  _{Y(0,\infty )}
\leq C_{2} \|g\|_{X(0,\infty )}
\quad
\mbox{for every nonincreasing }g\in \mathcal{M}_{+}(0,\infty )
\end{eqnarray}
is valid with some positive constant $C_{2}$. This result originated as
a consequence \citep[Corollary~9.8]{CPS} of a more general principle
established in~\citep[Theorem~9.5]{CPS} in connection with sharp
higher-order Sobolev-type embeddings and its extension to unbounded
intervals was given in~\citep[Theorem~1.10]{P}.

An important corollary of the Hardy--Littlewood inequality~\reftext{\eqref{E:HL}}
is the fact that if $f$ is a~nonincreasing function on $(0,\infty )$ and
$X$ is an~r.i.~space over $(0,\infty )$, then in fact
one has
%
\begin{eqnarray}
\label{E:corollary-of-HL}
\|f\|_{X(0,\infty )}=\sup \left \{  \int _{0}^{\infty }g^{*}(t)f(t)\,
{\mathrm{d}}t: \|g\|_{X'(0,\infty )}\leq 1\right \}  .
\end{eqnarray}
In other words, for such $f$, the supremum can be reduced to
nonincreasing functions only without any loss of information. This fact
has deep consequences and will be used in the proofs below.

For each $a\in (0,\infty )$, let $D_{a}$ denote the
\textit{dilation operator} defined on every nonnegative measurable
function $f$ on $(0,\infty )$ by
\begin{eqnarray*}
(D_{a}f)(t)=f(at),\quad t\in (0,\infty ).
\end{eqnarray*}
The operator $D_{a}$ is bounded on every rearrangement-invariant
space over $(0,\infty )$ (hence in particular on the representation
space of any r.i.~space over an~arbitrary adequate
measure space). More precisely, if $X$ is any given
r.i.~space over $(0,\infty )$ with respect to the
one-dimensional Lebesgue measure, then we have
\begin{eqnarray*}
\|D_{a}f\|_{X}\leq C\|f\|_{X}, \quad f\in X,
\end{eqnarray*}
with $C\le \max\{1,\frac{1}{a}\}$. For more details,
see~\citep[Chapter~3, Proposition~5.11]{BS}.

Among basic examples of function norms are those associated with the
standard Lebesgue spaces $L^{p}$. For $p\in (0,\infty ]$, we define the
functional $\varrho _{p}$ by
\begin{eqnarray*}
\varrho _{p}(f)=\|f\|_{p}=
\Cases{
\left (\int _{R}f^{p}\,{\mathrm{d}}\mu \right )^{\frac{1}{p}}
&\textup{if}\ 0<p<\infty ,
\cr
\operatorname{ess\,\sup }_{R}f
&\textup{if}\ p=\infty
}
\end{eqnarray*}
for $f \in \mathcal{M}_{+}(R,\mu )$. If $p\in [1,\infty ]$, then
$\varrho _{p}$ is an~r.i.~norm.

If $0< p,q\le \infty $, we define the functional $\varrho _{p,q}$ by
\begin{eqnarray*}
\varrho _{p,q}(f)=\|f\|_{p,q}=
\left \|  s^{\frac{1}{p}-\frac{1}{q}}f
^{*}(s)\right \|  _{q}
\end{eqnarray*}
for $f \in \mathcal{M}_{+}(R,\mu )$. The set $L^{p,q}$, defined as the
collection of all $f\in \mathcal{M}(R,\mu )$ satisfying $\varrho _{p,q}(|f|)<
\infty $, is called a~\textit{Lorentz space}. If either $1<p<\infty $
and $1\leq q\leq \infty $ or $p=q=1$ or $p=q=\infty $, then
$\varrho _{p,q}$ is equivalent to an~r.i.~norm
in the sense that there exists an~r.i.~norm
$\sigma $ and a~constant $C$, $0<C<\infty $, depending on $p,q$ but
independent of $f$, such that
\begin{eqnarray*}
C^{-1}\sigma (f)\leq \varrho _{p,q}(f)\leq C\sigma (f).
\end{eqnarray*}
As a~consequence, $L^{p,q}$ is considered to be
an~r.i.~space for these cases of $p,q$,
see~\citep[Chapter~4]{BS}. If either $0<p<1$ or $p=1$ and $q>1$, then
$L^{p,q}$ is a~quasi-normed space. If $p=\infty $ and $q<\infty $, then
$L^{p,q}=\{0\}$. For every $p\in [1,\infty ]$, we have $L^{p,p}=L^{p}$.
Furthermore, if $p,q,r\in (0,\infty ]$ and $q\leq r$, then the inclusion
$L^{p,q}\subset L^{p,r}$ holds.

If ${\mathbb{A}}=[\alpha _{0},\alpha _{\infty }]\in \mathbb{R}^{2}$ and
$t\in \mathbb{R}$, then we shall use the notation ${\mathbb{A}}+t=[
\alpha _{0}+t,\alpha _{\infty }+t]$.

Let $0<p,q\le \infty $, ${\mathbb{A}}=[\alpha _{0},\alpha _{\infty }]
\in \mathbb{R}^{2}$ and ${\mathbb{B}}=[\beta _{0},\beta _{\infty }]
\in \mathbb{R}^{2}$. Then we define the functionals $
\varrho _{p,q;{\mathbb{A}}}$ and
$\varrho _{p,q;{\mathbb{A}},{\mathbb{B}}}$ on $\mathcal{M}_{+}(R,
\mu )$ by
\begin{eqnarray*}
\varrho _{p,q;{\mathbb{A}}}(f)=
\left \|  t^{\frac{1}{p}-\frac{1}{q}}
\ell ^{{\mathbb{A}}}(t) f^{*}(t)\right \|  _{L^{q}(0,\infty )}
\end{eqnarray*}
and
\begin{eqnarray*}
\varrho _{p,q;{\mathbb{A}},{\mathbb{B}}}(f)=
\left \|  t^{\frac{1}{p}-
\frac{1}{q}}\ell ^{{\mathbb{A}}}(t)\ell \ell ^{{\mathbb{B}}}(t)
f^{*}(t)\right \|  _{L^{q}(0,\infty )},
\end{eqnarray*}
where
\begin{eqnarray*}
\ell ^{{\mathbb{A}}}(t)=
\Cases{
(1-\log t)^{\alpha _{0}}
&\textup{if}\ t\in (0,1),
\cr
(1+\log t)^{\alpha _{\infty }}
&\textup{if}\ t\in [1,\infty )
}
\end{eqnarray*}
and
\begin{eqnarray*}
\ell \ell ^{{\mathbb{B}}}(t)=
\Cases{
(1+\log (1-\log t))^{\beta _{0}}
&\textup{if}\ t\in (0,1),
\cr
(1+\log (1+\log t))^{\beta _{\infty }}
&\textup{if}\ t\in [1,\infty ).
}
\end{eqnarray*}
The set $L^{p,q;{\mathbb{A}}}$, defined as the collection of all
$f\in \mathcal{M}(R,\mu )$ satisfying $\varrho _{p,q;{\mathbb{A}}}(|f|)<
\infty $, is called a~\textit{Lorentz--Zygmund space}, and the set
$L^{p,q;{\mathbb{A}},{\mathbb{B}}}$, defined as the collection of all
$f\in \mathcal{M}_{+}(R,\mu )$ satisfying $
\varrho _{p,q;{\mathbb{A}},{\mathbb{B}}}(|f|)<\infty $, is called
a~\textit{generalized Lorentz--Zygmund space}. The functions of the form
$\ell ^{{\mathbb{A}}}$, $\ell \ell ^{{\mathbb{B}}}$ are called
\textit{broken logarithmic functions}. The spaces of this type proved
to be quite useful since they provide a common roof for many customary
spaces. These include not only Lebesgue spaces and Lorentz spaces, but
also all types of exponential and logarithmic Zygmund classes, and also
the spaces discovered independently by Maz'ya (in a~somewhat implicit
form involving capacitary estimates~\citep[pp.~105 and~109]{Mabook}),
Hansson~\citep{Ha} and Br\'{e}zis--Wainger~\citep{BW} who used it to
describe the sharp target space in a limiting Sobolev embedding (the
spaces can be also traced in the works of Brudnyi~\citep{B} and, in a
more general setting, Cwikel and Pustylnik~\citep{CP}). One of the
benefits of using broken logarithmic functions consists in the fact that
the underlying measure space can be considered to have either finite or
infinite measure. For the detailed study of generalized Lorentz--Zygmund
spaces we refer the reader to~\citep{EOP,EOP-broken,OP, FS}.

We further define the spaces $L^{(p,q;{\mathbb{A}})}$ through the
functionals $\varrho _{(p,q;{\mathbb{A}})}$ given on $\mathcal{M}_{+}(R,
\mu )$ by
\begin{eqnarray*}
\varrho _{(p,q;{\mathbb{A}})}(f)=
\left \|  t^{\frac{1}{p}-\frac{1}{q}}
\ell ^{{\mathbb{A}}}(t) f^{**}(t)\right \|  _{L^{q}(0,\infty )}
\end{eqnarray*}
and, in an~analogous way, all the other spaces involving various levels
of logarithms.

Let $X$ and $Y$ be r.i.~spaces over possibly
different measure spaces $(R,\mu )$ and $(S,\nu )$, respectively, and
let $T$ be an operator defined on $X$ with values in $\mathcal{M}(S,
\nu )$. We say that $T$ is \textit{bounded} from $X$ to $Y$, a fact
which is denoted by $T\colon X\to Y$, if there exists a~positive
constant $C$ such that
\begin{eqnarray*}
\|Tf\|_{Y}\leq C\|f\|_{X} , \quad f\in X.
\end{eqnarray*}
In an~important special case when $T$ is the identity operator, we say
that $X$ is \textit{embedded} into $Y$ and write $X\hookrightarrow Y$.
If $T'$ is another operator defined at least on $Y'$ with values in
$\mathcal{M}(R,\mu )$ and such that
%
\begin{eqnarray}
\label{E:duality-of-operators}
\int _{R}(Tf) g\,{\mathrm{d}}\mu =\int _{S}f(T'g)\,\mathrm{d}\nu
\end{eqnarray}
for every $f\in X$ and $g\in Y'$, then $T\colon X\to Y$ is equivalent
to $T'\colon Y'\to X'$.

Let $P$ and $Q$ be the integral operators defined by
\begin{eqnarray*}
(Pf)(t)=\frac{1}{t}\int _{0}^{t} f(s)\,{\mathrm{d}}s,\quad t\in (0,\infty ),
\end{eqnarray*}
and
\begin{eqnarray*}
(Qf)(t)=\int _{t}^{\infty }f(s)\frac{{\mathrm{d}}s}{s}, \quad t\in (0,\infty ),
\end{eqnarray*}
for those functions on $f\in \mathcal{M}_{0}(0,\infty )$ for which the
respective integrals have sense. As an interchange of integration shows,
\begin{eqnarray*}
\int _{0}^{\infty }(Pf)(t)g(t)\,{\mathrm{d}}t=\int _{0}^{\infty }f(t)(Qg)(t)\,{\mathrm{d}}t,
\end{eqnarray*}
for all $f$ and $g$ for which the integrals make sense. Hence, the
operators $P$ and $Q$ are formally adjoint with respect to the
$L^{1}$-pairing and therefore satisfy a relation in the spirit
of~\reftext{\eqref{E:duality-of-operators}}. As a~consequence, one has the
equivalence
%
\begin{eqnarray}
\label{E:equivalence-of-P-Q}
P\colon X\to Y\quad \Leftrightarrow \quad Q\colon Y'\to X'
\end{eqnarray}
for every pair of r.i.~spaces $X,Y$ over
$(0,\infty )$ (with the same operator norm). Another important example
is that when $(R,\mu )$ is arbitrary and both $T$ and $T'$ are identity
operators. Then~\reftext{\eqref{E:duality-of-operators}} is trivially satisfied
and, as a consequence, one gets
%
\begin{eqnarray}
\label{E:equivalence-of-identities}
X\hookrightarrow Y\quad \Leftrightarrow \quad Y'\hookrightarrow X'
\end{eqnarray}
for every pair of r.i.~spaces $X,Y$, again with the
same embedding constant, see~\citep[Chapter~1, Proposition~2.10]{BS}.

We will say that an~r.i.~space $Y$ over $(S,\nu )$ is
a~\textit{range partner} for a~given r.i.~space
$X$ over $(R,\mu )$ with respect to a sublinear operator $T$ if
$T\colon X\to Y$. We say that $Y$ is the~\textit{optimal range partner}
for $X$ if one has $Y\hookrightarrow Z$ for every range partner $Z$ for
$X$ with respect to $T$. We analogously define a~\textit{domain partner}
and the \textit{optimal domain partner}, that is, the largest possible
domain space.

Throughout the paper the convention that $\frac{1}{\infty }=0$, and
$0\cdot \infty =0$ is used without further explicit reference. We write
$A\approx B$ when the ratio $A/B$ is bounded from below and from above
by positive constants independent of appropriate quantities appearing
in expressions $A$ and $B$.

\section{The Hardy-Littlewood maximal operator} \label{sec3}

\noindent
In this section, the relevant r.i.~spaces are
considered over $\mathbb{R}^{n}$ endowed with the $n$-dimensional
Lebesgue measure. The Lebesgue measure of a~measurable set $E\subset
\mathbb{R}^{n}$ will be denoted by $|E|$.

The \textit{Hardy--Littlewood maximal operator}, $M$, is defined for
every locally integrable function $f$ on $\mathbb{R}^{n}$ and every
$x\in \mathbb{R}^{n}$ by
\begin{eqnarray*}
Mf(x)=\sup _{Q\owns x}\frac{1}{|Q|}\int _{Q}|f(y)|\,{\mathrm{d}}y,
\end{eqnarray*}
where the supremum is extended over all cubes $Q\subset \mathbb{R}
^{n}$, whose edges are parallel to the coordinate axes of $\mathbb{R}
^{n}$, that contain $x$.

The operator $M$ is merely sublinear, rather than linear, and it is
clearly a contraction on $L^{\infty }$. On the other hand, $Mf$ is never
integrable unless $f\equiv 0$. For every locally-integrable function
$f$ on $\mathbb{R}^{n}$, one has $|f|\leq Mf$ almost everywhere. The
most important information (for our purpose) concerning the operator
$M$, now classical, states that there exist positive constants
$c,c'$, depending only on $n$, such that
%
\begin{eqnarray}
\label{E:herz}
c(Mf)^{*}(t)\leq f^{**}(t)\leq c'(Mf)^{*}(t), \quad t\in (0,\infty ),
\end{eqnarray}
for every locally integrable function $f$ on $\mathbb{R}^{n}$. The first
inequality in~\reftext{\eqref{E:herz}} was established during the 1930s in works
of R.M.~Gabriel~\citep{G}, F.~Riesz~\citep{R} and N.~Wiener~\citep{W},
while the second was added later through the efforts of
C.~Herz~\citep{H} (for one dimension) and C.~Bennett and
R.~Sharpley~\citep{BS-paper} (for higher dimensions). The result is
summarized and proved in~\citep[Chapter~3, Theorem~3.8]{BS}.

We shall now state the first principal result of this section, in which
we characterize the optimal range partner to a given space with respect
to the operator~$M$.

\begin{theorem}
\label{T:maximal-operator}
Let $X$ be an~r.i.~space over $\mathbb{R}^{n}$ such
that
%
\begin{eqnarray}
\label{E:psi-condition}
\psi \in X'(0,\infty ),
\end{eqnarray}
where $\psi (t)=\chi _{(0,1)}(t)\log \tfrac{1}{t}$, $t\in (0,\infty )$.
Define the functional $\sigma $ by
\begin{eqnarray*}
\sigma (f)=\left \|  \int _{t}^{\infty }f^{*}(s)\frac{{\mathrm{d}}s}{s}\right \|  _{X'(0,\infty )}, \quad f\in \mathcal{M}_{+}(\mathbb{R}
^{n}).
\end{eqnarray*}
Then $\sigma $ is an~r.i.~norm and
%
\begin{eqnarray}
\label{E:M-bounded}
M\colon X\to Y,
\end{eqnarray}
where $Y=Y(\sigma ')$. Moreover, $Y$ is the optimal
\textup{(}smallest\textup{)} r.i.~space for
which~\reftext{\eqref{E:M-bounded}} holds.

Conversely, if~\reftext{\eqref{E:psi-condition}} is not true, then there does not
exist an~r.i.~space $Y$ for which~\reftext{\eqref{E:M-bounded}}
holds.
\end{theorem}

We now turn our attention to the question of the optimal domain space when
the target space is prescribed. This situation is considerably simpler
than the reverse one as no associate norms need to be involved.

\begin{theorem}
\label{T:maximal-operator-domain}
Let $Y$ be an~r.i.~space over $\mathbb{R}^{n}$ such
that
%
\begin{eqnarray}
\label{E:psi-condition-domain}
\psi \in Y(0,\infty ),
\end{eqnarray}
where $\psi (t)=\min \{1,\frac{1}{t}\}$ for $t\in (0,\infty )$. Define
the functional $\varrho $ by
\begin{eqnarray*}
\varrho (f)=\|f^{**}\|_{Y(0,\infty )}, \quad f\in \mathcal{M}_{+}(
\mathbb{R}^{n}).
\end{eqnarray*}
Then $\varrho $ is an~r.i.~norm
and~\reftext{\eqref{E:M-bounded}} is satisfied, where $X=X(\varrho )$. Moreover,
$X$ is the optimal \textup{(}largest\textup{)} rearrangement-invariant
space for which~\reftext{\eqref{E:M-bounded}} holds.

Conversely, if~\reftext{\eqref{E:psi-condition-domain}} is not true, then there
does not exist an~r.i.~space $X$ for
which~\reftext{\eqref{E:M-bounded}} holds.
\end{theorem}

In our final result of this section we present a~collection of
nontrivial examples based on Lorentz--Zygmund spaces.

\begin{theorem}
\label{T:-maximal-operator-GLZ}
Let $p,q\in [1,\infty ]$, ${\mathbb{A}}\in \mathbb{R}^{2}$. Then
\begin{eqnarray}
M\colon L^{p,q; {\mathbb{A}}} \to
\left\{
\begin{array}{l@{\quad }l@{\quad }l}
L^{1, 1, \mathbb{A} - 1},   & p = 1, q = 1, \alpha _{0}\geq 1,
\alpha _{\infty }< -1, &\mathrm{(a)}
\label{E:maximal_p1}
\\
Y,
& p = 1, q = 1, \alpha _{0}\geq 1, -1\leq \alpha _{\infty }
\leq 0, & \mathrm{(b)}
\label{E:maximal_p2}
\\
L^{p,q;{\mathbb{A}}}, & 1< p <\infty ~ or\\
& p=\infty , 1\leq q<\infty , \alpha _{0} + \frac{1}{q} < 0~ or \\
& p=\infty , q=\infty , \alpha _{0}\leq 0,
\end{array}\right.
\end{eqnarray}
where $Y$ is the \textup{(}unique\textup{)} rearrangement-invariant
space whose associate space $Y'$ satisfies
\begin{eqnarray*}
\|f\|_{Y'}=\sup _{0<t<\infty }\ell ^{-{\mathbb{A}}}(t)\int _{t}^{\infty
}f^{*}(s)\,\frac{{\mathrm{d}}s}{s}, \quad f\in \mathcal{M}_{+}(\mathbb{R}^{n}).
\end{eqnarray*}
These spaces are the optimal range partners with respect to~$M$.
\end{theorem}

We note that the space $Y'$, given in terms of an operator-induced norm,
cannot be expressed in terms of a Lorentz--Zygmund norm.

We shall now proceed to prove the stated results.

\begin{proof}[Proof of \reftext{Theorem~\ref{T:maximal-operator}}]
The functional $\sigma $ is obviously rearrangement invariant and,
thanks to the Monotone Convergence Theorem, it satisfies the lattice
axiom and the Fatou axiom. From (P1), only the triangle inequality needs
proving. Let $f,g\in \mathcal{M}(\mathbb{R}^{n})$. By the definition of
the associate space, one has
\begin{eqnarray*}
\left \|  \int _{t}^{\infty }(f+g)^{*}(s)\frac{{\mathrm{d}}s}{s}\right \|  _{X'(0,\infty )}
=
\sup _{\|h\|_{X(0,\infty )}\leq 1}
\int _{0}^{\infty }h(t)\int _{t}^{\infty }(f+g)^{*}(s)\frac{{\mathrm{d}}s}{s}\,{\mathrm{d}}t.
\end{eqnarray*}
Since the function
\begin{eqnarray*}
t\mapsto \int _{t}^{\infty }(f+g)^{*}(s)\frac{{\mathrm{d}}s}{s}
\end{eqnarray*}
is nonincreasing on $(0,\infty )$, we in fact have
(cf.~\reftext{\eqref{E:corollary-of-HL}})
\begin{eqnarray*}
\left \|  \int _{t}^{\infty }(f+g)^{*}(s)\frac{{\mathrm{d}}s}{s}\right \|  _{X'(0,\infty )}
=
\sup _{\|h\|_{X(0,\infty )}\leq 1}
\int _{0}^{\infty }h^{*}(t)\int _{t}^{\infty }(f+g)^{*}(s)\frac{{\mathrm{d}}s}{s}\,{\mathrm{d}}t.
\end{eqnarray*}
Thus, by the Fubini theorem,
\begin{eqnarray*}
\left \|  \int _{t}^{\infty }(f+g)^{*}(s)\frac{{\mathrm{d}}s}{s}\right \|  _{X'(0,\infty )}
&=
\sup _{\|h\|_{X(0,\infty )}
\leq 1}\int _{0}^{\infty }(f+g)^{*}(s)h^{**}(s)\,{\mathrm{d}}s.
\end{eqnarray*}
By~\reftext{\eqref{E:subadditivity-of-doublestar}} and the Hardy lemma, one has,
for every such $h$,
\begin{eqnarray*}
\int _{0}^{\infty }(f+g)^{*}(s)h^{**}(s)\,{\mathrm{d}}s\leq \int _{0}^{\infty }f^{*}(s)h^{**}(s)\,{\mathrm{d}}s+\int _{0}^{\infty }g^{*}(s)h^{**}(s)\,{\mathrm{d}}s.
\end{eqnarray*}
This estimate, combined with the preceding identity and the
subadditivity of the supremum, finally yields
\begin{eqnarray*}
\left \|  \int _{t}^{\infty }(f+g)^{*}(s)\frac{{\mathrm{d}}s}{s}\right \|  _{X'(0,{\infty })}
\leq \left \|  \int _{t}^{\infty
}f^{*}(s)\frac{{\mathrm{d}}s}{s}\right \|  _{X'(0,{\infty })}
+
\left \|  \int _{t}^{\infty }g
^{*}(s)\frac{{\mathrm{d}}s}{s}\right \|  _{X'(0,{\infty })},
\end{eqnarray*}
establishing the triangle inequality for $\sigma $.

As for (P4), let $E\subset \mathbb{R}^{n}$ be a set of finite measure.
We need to prove that
\begin{eqnarray*}
\left \|  \int _{t}^{\infty }\chi _{E}^{*}(s)\frac{{\mathrm{d}}s}{s}\right \|  _{X'(0,\infty )}<\infty .
\end{eqnarray*}
Since $\chi _{E}^{*}=\chi _{(0,|E|)}$, this amounts to showing the
finiteness of the quantity
\begin{eqnarray*}
\left \|  \chi _{(0,|E|)}(t)\int _{t}^{|E|}\frac{{\mathrm{d}}s}{s}\right \|  _{X'(0,\infty )}=\left \|  \chi _{(0,|E|)}(t)\log
\tfrac{|E|}{t}\right \|  _{X'(0,\infty )}.
\end{eqnarray*}
As $D_{|E|}(\chi _{(0,|E|)}(t)\log \tfrac{|E|}{t})=\chi _{(0,1)}(t)
\log \tfrac{1}{t}$, and the dilation operator $D_{|E|}$ is bounded on
$X'(0,\infty )$, we obtain that $\left \|  \chi _{(0,|E|)}(t)\log
\tfrac{|E|}{t}\right \|  _{X'(0,\infty )}$ is finite if and only
if~\reftext{\eqref{E:psi-condition}} holds, which, however, is guaranteed by the
assumption. This shows (P4).

Finally, to verify (P5), let $f\in \mathcal{M}_{+}(\mathbb{R}^{n})$ and
let $E\subset \mathbb{R}^{n}$ be of finite measure. Then, by the
monotonicity of $f^{*}$, we obtain
\begin{eqnarray*}
\sigma (f)
&
\geq \left \|  \int _{t}^{2t}f^{*}(s)\frac{{\mathrm{d}}s}{s}\right \|  _{X'(0,\infty )}\geq \left \|  f^{*}(2t)\int _{t}
^{2t}\frac{{\mathrm{d}}s}{s}\right \|  _{X'(0,\infty )}
=
\left \|  f^{*}(2t)\right \|
_{X'(0,\infty )}\log 2.
\end{eqnarray*}
Since $X'$ itself is an~r.i.~space, it satisfies (P5).
In other words, there is a positive constant $C_{E}$,
independent of $f$, such that
\begin{eqnarray*}
\int _{E}f\,{\mathrm{d}}\mu \leq C_{E}\left \|  f\right \|  _{X'}.
\end{eqnarray*}
By the rearrangement invariance of the space $X'$ and the boundedness
of the dilation operator on $X'(0,\infty )$, we finally get from the
preceding estimates that
\begin{eqnarray*}
\int _{E}f\,{\mathrm{d}}\mu \leq C_{E}\left \|  f^{*}\right \|  _{X'(0,\infty )}\leq
\widetilde{C}_{E}\left \|  f^{*}(2t)\right \|  _{X'(0,\infty )}
\leq \frac{\widetilde{C}_{E}}{\log 2}\sigma (f)
\end{eqnarray*}
for some positive constant $\widetilde{C}_{E}$,
independent of $f$. This shows that $\sigma $ satisfies (P5) and,
altogether, that $\sigma $ is an~r.i.~norm.

We shall now show that $M\colon X\to Y$. Recall that
\begin{eqnarray*}
\left \|  \int _{t}^{\infty }g^{*}(s)\frac{{\mathrm{d}}s}{s}\right \|  _{X'(0,{\infty })}=\|g\|_{Y'(0,{\infty })}, \quad
g\in \mathcal{M}_{+}(0,{\infty }).
\end{eqnarray*}
The next step is getting rid of the star in the last identity, which
can be done thanks to the equivalence of \reftext{\eqref{T:Lenka-unrestricted}}
and \reftext{\eqref{T:Lenka-nonincreasing}}. We conclude that there exists
a~positive constant $C$ such that,
\begin{eqnarray*}
\left \|  \int _{t}^{\infty }g(s)\frac{{\mathrm{d}}s}{s}\right \|  _{X'(0,{\infty })}\leq C\|g\|_{Y'(0,{\infty })}
, \quad g\in \mathcal{M}_{+}(0,{\infty }).
\end{eqnarray*}
We emphasize that this step (the fall of a star) is quite deep and that
it does not follow from the Hardy--Littlewood inequality (as it might
deceptively appear) because the integration takes place far away from
zero. Once the inequality is unrestricted to monotone functions, we are
entitled to apply the~standard argument using associate spaces.
Using~\reftext{\eqref{E:equivalence-of-P-Q}}, we get
\begin{eqnarray*}
\left \|  \frac{1}{t}\int _{0}^{t}g(s)\,{\mathrm{d}}s\right \|  _{Y(0,{\infty })}\leq C\|g\|_{X(0,{\infty })}, \quad
g\in \mathcal{M}_{+}(0,{\infty }),
\end{eqnarray*}
with the constant $C$ undamaged. Now we need our star back, but this
time that is achieved easily. We just restrict the last inequality to
the cone of nonincreasing functions and obtain
\begin{eqnarray*}
\left \|  \frac{1}{t}\int _{0}^{t}g^{*}(s)\,{\mathrm{d}}s\right \|  _{Y(0,{\infty })}\leq C\|g^{*}\|_{X(0,{\infty })}
, \quad g\in \mathcal{M}_{+}(0,{\infty }).
\end{eqnarray*}
Applying the rearrangement invariance of the space $X$ and using the
correspondence between an~r.i.~space and its
representation space, we readily see that this can be rewritten as
\begin{eqnarray*}
\left \|  \frac{1}{t}\int _{0}^{t}f^{*}(s)\,{\mathrm{d}}s\right \|  _{Y(0,{\infty })}\leq C\|f\|_{X}, \quad
f\in \mathcal{M}(\mathbb{R}^{n}).
\end{eqnarray*}
By the first inequality in~\reftext{\eqref{E:herz}}, we obtain that there exists
a~positive constant $C'$ such that
\begin{eqnarray*}
\left \|  (Mf)^{*}\right \|  _{Y(0,{\infty })}\leq C'\|f\|_{X}, \quad
\mathcal{M}(\mathbb{R}^{n}).
\end{eqnarray*}
Finally, the rearrangement invariance of the space $Y$ yields
\begin{eqnarray*}
\left \|  Mf\right \|  _{Y}\leq C'\|f\|_{X}, \quad \mathcal{M}(\mathbb{R}^{n}).
\end{eqnarray*}
In other words, $M\colon X\to Y$.

We shall now establish the optimality property of $Y$. To this end,
assume that, for some r.i.~space $Z$ over
$\mathbb{R}^{n}$, we have $M\colon X\to Z$. This means that there exists
a~positive constant $C$ such that for every $f\in L^{1}_{
\operatorname{loc}}(\mathbb{R}^{n})$ the inequality
\begin{eqnarray*}
\|Mf\|_{Z}\leq C\|f\|_{X}
\end{eqnarray*}
holds. Translated to the world of rearrangements, this reads
\begin{eqnarray*}
\|(Mf)^{*}\|_{Z(0,\infty )}\leq C\|f^{*}\|_{X(0,\infty )}.
\end{eqnarray*}
Using the second inequality in~\reftext{\eqref{E:herz}}, we get
\begin{eqnarray*}
\left \|  \frac{1}{t}\int _{0}^{t}f^{*}(s)\,{\mathrm{d}}s\right \|  _{Z(0,{\infty })}\leq C'\|f^{*}\|_{X(0,{\infty })}
, \quad \mathcal{M}(\mathbb{R}
^{n}).
\end{eqnarray*}
with some positive constant $C'$. A special case of the
Hardy--Littlewood inequality together with (P2) for $Z$ now yields
\begin{eqnarray*}
\left \|  \frac{1}{t}\int _{0}^{t}g(s)\,{\mathrm{d}}s\right \|  _{Z(0,{\infty })}
\leq \left \|  \frac{1}{t}\int _{0}
^{t}g^{*}(s)\,{\mathrm{d}}s\right \|  _{Z(0,{\infty })} , \quad g\in
\mathcal{M}_{+}(0,{\infty }).
\end{eqnarray*}
Thus, since $\|g\|_{X(0,\infty )}=\|g^{*}\|_{X(0,\infty )}$, we have
\begin{eqnarray*}
\left \|  \frac{1}{t}\int _{0}^{t}g(s)\,{\mathrm{d}}s\right \|  _{Z(0,{\infty })}\leq C'\|g\|_{X(0,{\infty })}, \quad
g\in \mathcal{M}_{+}(0,{\infty }).
\end{eqnarray*}
By~\reftext{\eqref{E:equivalence-of-P-Q}}, this is nothing else than
\begin{eqnarray*}
\left \|  \int _{t}^{\infty }g(s)\frac{{\mathrm{d}}s}{s}\right \|  _{X'(0,{\infty })}\leq C'\|g\|_{Z'(0,{\infty })}
, \quad g\in \mathcal{M}_{+}(0,{\infty }).
\end{eqnarray*}
Restricting this inequality to nonincreasing functions, we get
\begin{eqnarray*}
\left \|  \int _{t}^{\infty }g^{*}(s)\frac{{\mathrm{d}}s}{s}\right \|  _{X'(0,{\infty })}\leq C'\|g^{*}\|_{Z'(0,{\infty })}
, \quad g\in \mathcal{M}_{+}(0,\infty ).
\end{eqnarray*}
By the definition of $Y'$ and by the rearrangement invariance of
$Z'(0,{\infty })$, this can be rewritten as
\begin{eqnarray*}
\left \|  g\right \|  _{Y'}\leq C'\|g\|_{Z'},
\quad g\in \mathcal{M}_{+}(\mathbb{R}^{n}).
\end{eqnarray*}
In other words, we have established the embedding $Z'\hookrightarrow
Y'$, which is, due to~\reftext{\eqref{E:equivalence-of-identities}}, equivalent
to $Y\hookrightarrow Z$. This shows that $Y$ is indeed the optimal range
partner for $X$ with respect to~$M$.

Finally, assume that $\psi \notin X'(0,\infty )$ and suppose that
$M\colon X\to Y$ for some $Y$. Then, following the same line of argument
as above, we obtain that
\begin{eqnarray*}
\|Qg^{*}\|_{X'(0,\infty )}\leq C\|g\|_{Y'(0,\infty )}, \quad g\in \mathcal{M}_{+}(0,
\infty ),
\end{eqnarray*}
with some $C$, $0<C<\infty $, independent of $g$. Inserting $g=\chi _{(0,1)}$, we obtain that the right side of the last
inequality is finite, since $Y'$ is an~r.i.~space,
and, as such, it must obey the axiom (P4). The left side is however
infinite, because we have
\begin{eqnarray*}
\|Q\chi _{(0,1)}^{*}\|_{X'(0,\infty )}=\|\psi \|_{X'(0,\infty )}=
\infty .
\end{eqnarray*}
This is absurd, hence there is no such $Y$. The proof is complete.
\end{proof}

\begin{proof}[Proof of \reftext{Theorem~\ref{T:maximal-operator-domain}}]
The functional $\varrho $ obviously obeys (P1), (P2), (P3) and (P6). In
particular, the triangle inequality follows immediately from the
triangle inequality for $Y(0,\infty )$
and~\reftext{\eqref{E:subadditivity-of-doublestar}}. Thanks to the boundedness of
the dilation operator on $Y(0,\infty )$, (P4) is equivalent to
$\chi _{(0,1)}^{**}\in Y(0,\infty )$, which is however guaranteed by the
assumption of the theorem, since $\chi _{(0,1)}^{**}=\psi $. Finally,
(P5) follows easily from the chain
\begin{eqnarray*}
\varrho (g)\geq \|g^{**}\chi _{(0,|E|)}\|_{Y(0,\infty )}
\geq g^{**}(|E|)
\|\chi _{(0,|E|)}\|_{Y(0,\infty )}
\geq \frac{1}{|E|}\|\chi _{(0,|E|)}\|_{Y(0,
\infty )}\int _{E} g(x)\,{\mathrm{d}}x,
\end{eqnarray*}
where $E\subset \mathbb{R}^{n}$ is an arbitrary set of finite measure
and $g\in \mathcal{M}_{+}(\mathbb{R}^{n})$. We used the monotonicity of
$g^{**}$ and the Hardy--Littlewood inequality. The operator $M$ is
obviously bounded from $X$ to $Y$ thanks to~\reftext{\eqref{E:herz}}. The
optimality of $X$ follows from the following simple argument. Suppose
that $M\colon Z\to Y$ for some r.i.~space $Z$. Then
$\|Mf\|_{Y}\leq C\|f\|_{Z}$ for some $C>0$ and all $f\in Z$. Therefore,
by~\reftext{\eqref{E:herz}} once again, we have $\|f^{**}\|_{Y}\leq C\|f\|_{Z}$,
which, however, is nothing else than the embedding $Z\hookrightarrow
X$. Finally, if $\psi \notin Y(0,\infty )$ then there is no domain
partner for $Y$ with respect to $M$, because if there was one, say
$X$, then one would have in particular $\|\chi _{(0,1)}^{**}\|_{Y(0,
\infty )}\leq C\|\chi _{(0,1)}\|_{X(0,\infty )}$, but the right-hand side
is finite due to (P4) for $X$ and the left-hand side is equal to
infinity since $\psi \notin Y(0,\infty )$. The proof is complete.
\end{proof}

\begin{proof}[Proof of \reftext{Theorem~\ref{T:-maximal-operator-GLZ}}]
We first recall that if for an~r.i.~space $X$ one has
$M\colon X\to X$, then automatically $X$ is the optimal range (and
domain) partner for itself with respect to $M$. This immediately follows
from the inequality $f^{**}\geq f^{*}$ combined with~\reftext{\eqref{E:herz}}.
Now~\citep[Theorem~3.8]{OP} together with~\reftext{\eqref{E:herz}} implies that
$M\colon L^{p,q;{\mathbb{A}}}\to L^{p,q;{\mathbb{A}}}$ when either
$1<p<\infty $ or $p=\infty $, $1\leq q<\infty $ and $\alpha _{0}+
\frac{1}{q}<0$ or $p=\infty $, $q=\infty $ and $\alpha _{0}\leq 0$. This
proves the assertion in all cases except (\ref{E:maximal_p1}a) and
(\ref{E:maximal_p2}b).

Assume now that $p=1$, $q=1$, $\alpha _{0}\geq 1$ and $\alpha _{\infty
}\leq 0$. By~\citep[Theorem~7.1]{OP}, $L^{1,1;{\mathbb{A}}}$ is
equivalent to an~r.i.~space. Moreover,
by~\citep[Theorem~6.6]{OP}, $(L^{1,1;{\mathbb{A}}})'=L^{\infty ,
\infty ;-{\mathbb{A}}}$. Thus, one has
\begin{eqnarray*}
\|\psi \|_{X'(0,\infty )}
\approx \sup _{0<t\leq 1}(1-\log t)^{1-\alpha
_{0}}(t)<\infty ,
\end{eqnarray*}
since $\alpha _{0}\geq 1$. In other words, $\psi \in X'(0,\infty )$.
Consequently, by \reftext{Theorem~\ref{T:maximal-operator}}, the optimal range
partner $Y$ for $L^{1,1;{\mathbb{A}}}$ with respect to~$M$ satisfies
%
\begin{eqnarray}
\label{E:Y'}
\|f\|_{Y'}=\left \|  \int _{t}^{\infty }f^{*}(s)\frac{{\mathrm{d}}s}{s}\right \|  _{X'(0,\infty )}
=\sup _{0<t<\infty }\ell ^{-{\mathbb{A}}}(t)
\int _{t}^{\infty }f^{*}(s)\,\frac{{\mathrm{d}}s}{s}, \quad f\in \mathcal{M}_{+}(\mathbb{R}^{n}).
\end{eqnarray}
This establishes (\ref{E:maximal_p2}b).

It remains to prove (\ref{E:maximal_p1}a). To do this we have to show
that, for this choice of parameters, the space $Y$ whose associate space
has norm given by~\reftext{\eqref{E:Y'}} coincides with $L^{1,1;{\mathbb{A}}-1}$.
We have
\begin{eqnarray*}
\|f\|_{Y'}
&=&\sup _{0<t<\infty }\ell ^{-{\mathbb{A}}}(t)\int _{t}^{
\infty }f^{*}(s)\,\frac{{\mathrm{d}}s}{s}
\\
&=&\sup _{0<t<\infty }\ell ^{-{\mathbb{A}}}(t)\int _{t}^{\infty }f^{*}(s)
\ell ^{-{\mathbb{A}}+1}(s)\ell ^{{\mathbb{A}}-1}(s)\frac{{\mathrm{d}}s}{s}
\\
&\leq &\left (\sup _{0<s<\infty }f^{*}(s)\ell ^{-{\mathbb{A}}+1}(s)\right )
\left (\sup _{0<t<\infty }\ell ^{-{\mathbb{A}}}(t)\int _{t}^{\infty }
\ell ^{{\mathbb{A}}-1}(s)\frac{{\mathrm{d}}s}{s}\right )
\\
&\approx &\|f\|_{L^{\infty ,\infty ; -{\mathbb{A}}+1}},
\end{eqnarray*}
and, conversely,
\begin{eqnarray*}
\|f\|_{Y'}
&\geq &\max \left \{
\sup _{0<t<1}(1-\log t)^{-\alpha _{0}}
\int _{t}^{\sqrt{t}}f^{*}(s)\,\frac{\mathrm{
d}s}{s},
\sup _{1<t<\infty }(1+\log t)^{-\alpha _{\infty }}\int _{t}
^{t^{2}}f^{*}(s)\,\frac{\mathrm{
d}s}{s}
\right \}
\\
&\geq &\max \left \{
\sup _{0<t<1}(1\!-\!\log t)^{-\alpha _{0}}f^{*}(
\sqrt{t})\log (t^{-\frac{1}{2}}),
\sup _{1<t<\infty }(1\!+\!\log t)^{-
\alpha _{\infty }}f^{*}(t^{2})\log t
\right \}
\\
&\approx &\max \left \{
\sup _{0<t<1}(1-\log t)^{1-\alpha _{0}}f^{*}(\sqrt{t}),
\sup _{1<t<\infty }(1+\log t)^{1-
\alpha _{\infty }}f^{*}(t^{2})
\right \}
\\
&\approx &\max \left \{
\sup _{0<t<1}(1-\log t)^{1-\alpha _{0}}f
^{*}(t),
\sup _{1<t<\infty }(1+\log t)^{1-\alpha _{\infty }}f^{*}(t)
\right \}
\\
&\approx &\|f\|_{L^{\infty ,\infty ; -{\mathbb{A}}+1}}.
\end{eqnarray*}
Therefore, $Y'=L^{\infty ,\infty ; -{\mathbb{A}}+1}$, and, finally,
by~\citep[Theorem~6.2]{OP}, we get $Y=L^{1,1; {\mathbb{A}}-1}$, as
desired.
\end{proof}

\section{The fractional maximal operator} \label{sec4}

\noindent
In this section we shall treat the~\textit{fractional maximal operator}
$M_{\gamma }$, defined for a fixed $\gamma \in (0,n)$ and for every
locally integrable function on $\mathbb{R}^{n}$ by
\begin{eqnarray*}
M_{\gamma }f(x)=\sup _{Q\owns x}\frac{1}{|Q|^{1-\frac{\gamma }{n}}}\int
_{Q}|f(y)|\,\mathrm{
d}y, \quad x\in \mathbb{R}^{n}.
\end{eqnarray*}

The operator $M_{\gamma }$ can be defined in the same way also for
$\gamma =0$, in which case it coincides with the Hardy--Littlewood
maximal operator, and constitutes thereby its natural generalization.
The two types of operators nevertheless have to be treated separately
because their behaviour in cases $\gamma =0$ and $\gamma >0$ is, rather
surprisingly, substantially different, and, in the fractional case, a
new approach involving a~specific supremum operator is needed for the
study of the optimal action of the operator on function spaces. Since
the supremum operator is not linear, the use of techniques based on
associate norms and spaces is somewhat limited, and a certain care has
to be exercised.

The result of~\citep[Theorem~1.1]{CKOP} shows that there exists
a~positive constant $C$ depending only on $\gamma $ and $n$ such that,
for every $\mathcal{M}(\mathbb{R}^{n})$, one has
%
\begin{eqnarray}
\label{E:upper-bound-for-fractional}
(M_{\gamma }f)^{*}(t)\leq C\sup _{t\leq s<\infty }s^{\frac{\gamma }{n}}
f^{**}(s) , \quad t\in (0,\infty ),
\end{eqnarray}
and, conversely, for every nonincreasing function $g$ on $(0,\infty )$
there exists some $f_{0}\in L^{1}_{\operatorname{loc}}(\mathbb{R}^{n})$
such that $f_{0}^{*}=g$ almost everywhere on $(0,\infty )$ and
%
\begin{eqnarray}
\label{E:lower-bound-for-fractional}
(M_{\gamma }f_{0})^{*}(t)\geq c\sup _{t\leq s<\infty }s^{\frac{\gamma
}{n}} g^{**}(s) , \quad t\in (0,\infty ),
\end{eqnarray}
where, again, $c$ is some positive constant which depends only on
$\gamma $ and $n$. For $\gamma =0$, the combination
of~\reftext{\eqref{E:lower-bound-for-fractional}}
and~\reftext{\eqref{E:upper-bound-for-fractional}} coincides with~\reftext{\eqref{E:herz}},
since the function $g^{**}$ is nonincreasing on $(0,\infty )$ for any
$g$.

\begin{theorem}
\label{T:fractional-maximal-operator}
Let $X$ be an~r.i.~space over $\mathbb{R}^{n}$. Let
$\gamma \in (0,n)$ and assume that
%
\begin{eqnarray}
\label{E:fund}
\inf _{1\leq t<\infty } \varphi _{X}(t)t^{-\frac{\gamma }{n}}>0.
\end{eqnarray}
Define the functional $\sigma $ by
%
\begin{eqnarray}
\label{E:sigma-frac}
\sigma (f)=\sup _{\atopfrac{h\sim f}{h\ge 0}}
\left \|  \int _{t}^{\infty }h(s)s^{\frac{\gamma }{n}-1}\,
\mathrm{
d}s\right \|  _{X'(0,\infty )},
\ f\in \mathcal{M}_{+}(\mathbb{R}^{n}),
\end{eqnarray}
where the supremum is taken over all $h\in \mathcal{M}_{+}(\mathbb{R}
^{n})$ equimeasurable with $f$. Then $\sigma $ is
an~r.i.~norm and
%
\begin{eqnarray}
\label{E:MG}
M_{\gamma }\colon X\to Y,
\end{eqnarray}
where $Y=Y(\sigma ')$. Moreover, $Y$ is the optimal
\textup{(}smallest\textup{)} r.i.~space for
which~\reftext{\eqref{E:MG}} holds.

Conversely, if~\reftext{\eqref{E:fund}} is not true, then there does not exist
an~r.i.~space $Y$ for which~\reftext{\eqref{E:MG}} holds.
\end{theorem}

The expression for the functional $\sigma $ in
\reftext{Theorem~\ref{T:fractional-maximal-operator}} is somewhat implicit. Our
next result however shows that it can be considerably simplified at a
relatively low cost. We shall need a \textit{supremum operator}. For a
fixed $\alpha \geq 0$, define the operator $T_{\alpha }$ on
$\mathcal{M}(0,\infty )$ by
\begin{eqnarray*}
T_{\alpha }f(t)= t^{-\alpha }\sup _{t\leq s<\infty }s^{\alpha }f^{*}(s),
\ t\in (0,\infty ).
\end{eqnarray*}

\begin{theorem}
\label{T:fractional-corollary}
Let $0<\gamma <n$ and let $X$ be an~r.i.~space over
$\mathbb{R}^{n}$. Assume that
%
\begin{eqnarray}
\label{E:bundedness-of-T}
T_{\frac{\gamma }{n}}\colon X(0,\infty )\to X(0,\infty ).
\end{eqnarray}
Define the functional $\tau $ by
\begin{eqnarray*}
\tau (f)=\sup _{\|h\|_{X(0,\infty )}\leq 1}\int _{0}^{\infty }f^{*}(s)(PT
_{\frac{\gamma }{n}}h)(s)s^{\frac{\gamma }{n}}\,\mathrm{
d}s.
\end{eqnarray*}
Then $\tau $ is an~r.i.~norm such that
\begin{eqnarray*}
M_{\gamma }\colon X\to Y,
\end{eqnarray*}
where $Y=Y(\tau ')$, and $Y$ is the optimal \textup{(}smallest\textup{)}
r.i.~space for which~\reftext{\eqref{E:MG}} holds. Moreover,
$\tau $ is equivalent to the functional
%
\begin{eqnarray}
\label{E:sigma-frac-corollary}
f\mapsto \left \|  \int _{t}^{\infty }f^{*}(s)s^{\frac{\gamma }{n}-1}
\,\mathrm{
d}s\right \|  _{X'(0,\infty )}, \quad f\in \mathcal{M}_{+}(\mathbb{R}^{n}).
\end{eqnarray}
\end{theorem}

\begin{remark}
\label{R:comparison}
The assumption~\reftext{\eqref{E:bundedness-of-T}} of
\reftext{Theorem~\ref{T:fractional-corollary}} is natural in view of the fact that
the classical endpoint mapping properties for the fractional maximal
operator $M_{\frac{\gamma }{n}}$ are of the form
\begin{eqnarray*}
M_{\frac{\gamma }{n}}\colon L^{1}\to L^{\frac{n}{n-\gamma },\infty }
\qquad
\textup{and}
\qquad
M_{\frac{\gamma }{n}}\colon L^{\frac{n}{\gamma },\infty }\to L^{
\infty },
\end{eqnarray*}
while those of $T_{{\frac{\gamma }{n}}}$ are
(cf.~\cite{T2,T3,GOP})
\begin{eqnarray*}
T_{{\frac{\gamma }{n}}}\colon L^{1}\to L^{1}
\qquad
\textup{and}
\qquad
T_{{\frac{\gamma }{n}}}\colon L^{\frac{n}{\gamma },\infty }\to L^{\frac{n}{
\gamma },\infty }.
\end{eqnarray*}
On the other hand,~\reftext{\eqref{E:bundedness-of-T}} is
\textit{strictly stronger} than~\reftext{\eqref{E:fund}}. Indeed, assume
that~\reftext{\eqref{E:bundedness-of-T}} is satisfied. Then, in particular, there
exists a~positive constant, $K$, such that for every $a\geq 1$ one has
\begin{eqnarray*}
\|T_{\frac{\gamma }{n}}\chi _{(0,a)}\|_{X(0,\infty )}
\leq K
\|\chi
_{(0,a)}\|_{X(0,\infty )}.
\end{eqnarray*}
Since
\begin{eqnarray*}
T_{\frac{\gamma }{n}}\chi _{(0,a)}(t)
=
\chi _{(0,a)}(t)a^{\frac{\gamma
}{n}}t^{-\frac{\gamma }{n}}\quad \textup{for}\ t\in (0,\infty ),
\end{eqnarray*}
we in fact have
\begin{eqnarray*}
\|\chi _{(0,a)}(t)a^{\frac{\gamma }{n}}t^{-\frac{\gamma }{n}}\|_{X(0,
\infty )} \leq K \varphi _{X}(a).
\end{eqnarray*}
Consequently,
\begin{eqnarray*}
\varphi _{X}(a)a^{-\frac{\gamma }{n}}
\geq K^{-1}\|\chi _{(0,a)}(t)t
^{-\frac{\gamma }{n}}\|_{X(0,\infty )}
\geq K^{-1}\|\chi _{(0,1)}(t)t
^{-\frac{\gamma }{n}}\|_{X(0,\infty )}.
\end{eqnarray*}
Hence
\begin{eqnarray*}
\inf _{1\leq a<\infty }\varphi _{X}(a)a^{-\frac{\gamma }{n}}\geq K^{-1}
\|\chi _{(0,1)}(t)t^{-\frac{\gamma }{n}}\|_{X(0,\infty )}>0,
\end{eqnarray*}
and~\reftext{\eqref{E:fund}} follows. This shows the implication \reftext{\eqref{E:bundedness-of-T}}$\Rightarrow $\reftext{\eqref{E:fund}}. The fact that
this implication cannot be reversed follows on considering $X=L^{\frac{n}{
\gamma },q}$ with $q\in [1,\infty )$. Every such space obviously
satisfies~\reftext{\eqref{E:fund}}, but it follows from~\citep[Theorem~3.2]{GOP}
that the operator $T_{{\frac{\gamma }{n}}}$ is not bounded on it,
hence~\reftext{\eqref{E:bundedness-of-T}} does not hold.
\end{remark}

For the optimal domain for the fractional maximal operator, we have the
following result. Its proof is analogous to that of
\reftext{Theorem~\ref{T:maximal-operator-domain}} and therefore is omitted.

\begin{theorem}
Let $0<\gamma <n$ and let $Y$ be an~r.i.~space over
$\mathbb{R}^{n}$ such that
%
\begin{eqnarray}
\label{E:fmo-domain-condition}
\psi \in Y(0,\infty ),
\end{eqnarray}
where $\psi (t)=(1+t)^{\frac{\gamma }{n}-1}$, $t\in (0,\infty )$. Define
the functional $\sigma $ by
\begin{eqnarray*}
\sigma (f)=\left \|  t^{\frac{\gamma }{n}}f^{**}(t) \right \|  _{Y(0,
\infty )}, \quad f\in \mathcal{M}_{+}(\mathbb{R}^{n}).
\end{eqnarray*}
Then $\sigma $ is an~r.i.~norm and
%
\begin{eqnarray}
\label{E:fractional-bounded-domain}
M_{\gamma }\colon X\to Y,
\end{eqnarray}
where $X=X(\sigma )$. Moreover, $X$ is the optimal
\textup{(}largest\textup{)} r.i.~space for
which~\reftext{\eqref{E:fractional-bounded-domain}} holds.

Conversely, if~\reftext{\eqref{E:fmo-domain-condition}} is not true, then there
does not exist an~r.i.~space $X$ for
which~\reftext{\eqref{E:fractional-bounded-domain}} holds.
\end{theorem}

Our next aim is to present an array of results concerning the optimal
range partners for Lorentz-Zygmund spaces of the form
$L^{p,q;{\mathbb{A}}}$ with respect to $M_{\gamma }$. Mapping properties
of $M_{\gamma }$ on Lorentz--Zygmund spaces were studied in~\citep{EO},
where the following results were established:
\begin{eqnarray*}
M_{\gamma }\colon L^{p,q; {\mathbb{A}}} \to
\left\{
\begin{array}{l@{\quad }l}
L^{{\frac{n}{n-\gamma }}, 1;{\mathbb{A}}-1},
& p=1, q=1,
\alpha _{0}\ge 0, \alpha _{\infty }< 0, \\
L^{{\frac{n}{n-\gamma }}, \infty ; {\mathbb{A}}},
& p=1,
q=1, \alpha _{0} \geq 0, \alpha _{\infty }\leq 0, \\
L^{\frac{np}{n-\gamma p},q;{\mathbb{A}}},
& 1<p<\tfrac{n}{
\gamma }, 1\le q \le \infty , \\
L^{\infty ,q;{\mathbb{A}}- \frac{1}{q}},
& p=\tfrac{n}{\gamma
}, 1\le q \leq \infty , \alpha _{0} < 0, \alpha _{\infty }> 0.
\end{array}\right.
\end{eqnarray*}

Our result concerning optimal range spaces for Lorentz--Zygmund spaces
reads as follows.

\begin{theorem}
\label{T:fractional-maximal-operator-GLZ}
Let $\gamma \in (0,n)$, $p,q\in [1,\infty ]$, ${\mathbb{A}}\in
\mathbb{R}^{2}$. Then
\begin{eqnarray}
M_{\gamma }\colon L^{p,q; {\mathbb{A}}} \to
\left\{
\begin{array}{l@{\quad }l@{\quad }l}
Y_{1},
& p = 1, q = 1, \alpha _{0}\geq 0, \alpha _{\infty }\leq 0, &\mathrm{(a)} \label{E:fractional_p1} \\
L^{\frac{np}{n-\gamma p}, q; {\mathbb{A}}}, & 1<p<\tfrac{n}{\gamma }, & \mathrm{(b)} \label{E:fractional_easy} \\
L^{\infty , \infty ; {\mathbb{A}}}, & p = \tfrac{n}{\gamma }, q=\infty , \alpha _{0}\leq 0, \alpha _{\infty }\geq 0, &\mathrm{(c)}
\label{E:fractional_infty1}
\\
Y_{2},
& p = \tfrac{n}{\gamma }, 1\le q < \infty , \alpha _{\infty
}\ge 0 ~ or & \mathrm{(d)}
\\
& p = \tfrac{n}{\gamma }, q =\infty , \alpha _{0} > 0,
\alpha _{\infty }\ge 0, & \mathrm{(e)}
\label{E:fractional_infty2}
\end{array}\right.
\end{eqnarray}
where $Y_{1}$ and $Y_{2}$ are the \textup{(}unique\textup{)}
r.i.~spaces whose associate spaces, $Y_{1}'$ and
$Y_{2}'$, satisfy
\begin{eqnarray*}
\|f\|_{Y_{1}'}
= \sup _{0<t<\infty } \ell ^{-{\mathbb{A}}}(t)
\int _{t}
^{\infty }f^{*}(s)s^{\frac{\gamma }{n}-1}\,\mathrm{
d}s,
\quad f \in \mathcal{M}_{+}(\mathbb{R}^{n}),
\end{eqnarray*}
and
\begin{eqnarray*}
\|f\|_{Y_{2}'}
= \sup _{\atopfrac{h\sim f}{h\ge 0}} \Bigl \|
t^{1-{\frac{\gamma }{n}}-\frac{1}{q'}}
\ell ^{-{\mathbb{A}}}(t)
\int _{t}^{\infty } h(s)\,s^{{\frac{\gamma }{n}}-1}
\,\mathrm{
d}s
\Bigr \|_{L^{q'}(0,\infty )},
\quad f \in \mathcal{M}_{+}(\mathbb{R}
^{n}),
\end{eqnarray*}
respectively. In particular, in the case ${\mathbb{A}}=[0,0]$, we have
$Y_{1}=L^{\frac{n}{n-\gamma },\infty }$ and $Y_{2}=L^{\infty }$.

Moreover, these spaces are the optimal range partners with respect to~$M
_{\gamma }$.
\end{theorem}

Again, there is no simpler way of characterizing the spaces
$Y_{1}'$ and $Y_{2}'$.

\begin{remark}
We note that the range spaces in
\reftext{Theorem~\ref{T:fractional-maximal-operator-GLZ}} essentially improve
those from~\citep{EO} when $p=q=1$, $\alpha _{0}\geq 0$, $\alpha _{
\infty }\leq 0$ and $|\alpha _{0}| + |\alpha _{\infty }| > 0$, and also
when $p=\frac{n}{\gamma }$, $1\leq q<\infty $, $\alpha _{0} < 0$ and
$\alpha _{\infty }> 0$. It is also worth noting that the spaces
$L^{\frac{n}{n-\gamma }, 1;{\mathbb{A}}-1}$ and $L^{\frac{n}{n-\gamma
}, \infty ; {\mathbb{A}}}$ are not comparable in the sense that neither
of them is contained in the other (see~\cite{OP} for details).
\end{remark}

Our next aim is to describe in more detail the relation between
\reftext{Theorems~\ref{T:fractional-maximal-operator}
and~\ref{T:fractional-corollary}}. \reftext{Theorem~\ref{T:fractional-corollary}}
asserts, among other statements, that in the particular cases
when~\reftext{\eqref{E:bundedness-of-T}} is satisfied, the
functionals~\reftext{\eqref{E:sigma-frac-corollary}} and~$\sigma $
from~\reftext{\eqref{E:sigma-frac}} are equivalent. We shall now point out an
interesting fact that the converse is also true, namely
if~\reftext{\eqref{E:bundedness-of-T}} is not satisfied, then the functional
in~\reftext{\eqref{E:sigma-frac-corollary}} is \textit{not equivalent} to
$\sigma $ from~\reftext{\eqref{E:sigma-frac}}. That, in fact, means that it is
\textit{essentially smaller} than $\sigma $. This is achieved through
the following result, which is definitely of independent interest and
maybe even a~little surprising.

\begin{theorem}
\label{T:lenka}
Assume that $X$ is an~r.i.~space over $\mathbb{R}
^{n}$ and $\gamma \in (0,n)$. Then the following statements are
equivalent:

\textup{(a)} $T_{\frac{\gamma }{n}}\colon X(0,\infty )\to X(0,\infty
)$,

\textup{(b)} there exists a~positive constant $C$ such that,
%
\begin{eqnarray}[ll]
\sup _{\atopfrac{h\sim f}{h\ge 0}}\left \|  \int _{t}^{\infty }h(s)s^{{\frac{\gamma }{n}}-1}\,
\mathrm{
d}s\right \|  _{X'(0,\infty )}\nonumber\\
\qquad \leq C
\left \|  \int _{t}^{\infty }f
^{*}(s)s^{{\frac{\gamma }{n}}-1}\,\mathrm{
d}s\right \|  _{X'(0,\infty )} , \quad f\in
\mathcal{M}_{+}(\mathbb{R}^{n}). \label{E:b}
\end{eqnarray}
\end{theorem}

\begin{remark}
\label{R:b}
We note that, since $f\sim f^{*}$, the converse inequality
to~\reftext{\eqref{E:b}}, namely
\begin{eqnarray*}
\left \|  \int _{t}^{\infty }f^{*}(s)s^{{\frac{\gamma }{n}}-1}\,\mathrm{
d}s\right \|  _{X'(0,\infty )}
\leq \sup _{\atopfrac{h\sim f}{h\ge 0}}\left \|  \int _{t}^{\infty }h(s)s^{{\frac{\gamma }{n}}-1}\,
\mathrm{
d}s\right \|  _{X'(0,\infty )},
\end{eqnarray*}
is trivial. In other words, if~\reftext{\eqref{E:b}} is true, then the two
quantities are in fact equivalent.
\end{remark}

In the proof of \reftext{Theorem~\ref{T:lenka}} we shall need the following
auxiliary result of independent interest.

\begin{lemma}
\label{L:lenka}
Assume that $I\colon (0,\infty )\to (0,\infty )$ is a~nondecreasing
function satisfying
%
\begin{eqnarray}
\label{E:I}
\int _{0}^{t} \frac{\mathrm{
d}s}{I(s)}
\approx \int _{t}^{2t} \frac{\mathrm{
d}s}{I(s)}
, \quad t\in (0,\infty ).
\end{eqnarray}
Let $N\in \mathbb{N}$, $0<t_{1}<\cdots <t_{N}<\infty $ and $a_{1},
\dots ,a_{N}>0$. Let
\begin{eqnarray*}
u=\sum _{i=1}^{N}a_{i}\chi _{(0,t_{i})}
\end{eqnarray*}
and let $X$ be an~r.i.~space over $(0,\infty)$. Then
\begin{eqnarray*}
\left \|  \int _{t}^{\infty } \frac{u(s)}{I(s)}\,\mathrm{
d}s \right \|  _{X(0,\infty )}
\approx \|v\|_{X(0,\infty )},
\end{eqnarray*}
where
\begin{eqnarray*}
v=\sum _{i=1}^{N}a_{i}\frac{t_{i}}{I(t_{i})}\chi _{(0,t_{i})}.
\end{eqnarray*}
\end{lemma}

\begin{proof}
First, we have
\begin{eqnarray*}
\int _{t}^{\infty }\frac{u(s)}{I(s)}\,ds
= \sum _{i=1}^{N}
\int _{t}^{
\infty } \frac{a_{i}\chi _{(0,t_{i})}(s)}{I(s)}\,\mathrm{
d}s
= \sum _{i=1}^{N}
a_{i}\chi _{(0,t_{i})}(t) \int _{t}^{t_{i}} \frac{
\mathrm{
d}s}{I(s)}, \quad t\in (0,\infty ).
\end{eqnarray*}
By \reftext{\eqref{E:I}},
\begin{eqnarray*}
\int _{0}^{t_{i}}\frac{\mathrm{
d}s}{I(s)}
\geq \int _{t}^{t_{i}}\frac{\mathrm{
d}s}{I(s)}
\geq \int _{\frac{t_{i}}{2}}^{t_{i}}\frac{\mathrm{
d}s}{I(s)}
\approx \int _{0}^{t_{i}}\frac{\mathrm{
d}s}{I(s)}, \quad t\in (0,\tfrac{t_{i}}{2}),
\end{eqnarray*}
whence
\begin{eqnarray*}
\sum _{i=1}^{N} a_{i}\chi _{(0,t_{i})}(t)\int _{t}^{t_{i}}\frac{\mathrm{
d}s}{I(s)}
& \geq &\sum _{i=1}^{N}
a_{i}\chi _{(0,\frac{t_{i}}{2})}(t)
\int _{t}^{t_{i}}\frac{\mathrm{
d}s}{I(s)}
\approx \sum _{i=1}^{N}
a_{i}\chi _{(0,\frac{t_{i}}{2})}(t)
\int _{0}^{t_{i}}\frac{\mathrm{
d}s}{I(s)}
\\
& \approx & \sum _{i=1}^{N}
a_{i}\frac{t_{i}}{I(t_{i})}
\chi _{(0,\frac{t_{i}}{2})}(t), \quad t\in (0,\infty
).
\end{eqnarray*}
Therefore, due to the boundedness of the dilation operator on
$X(0,\infty )$, we have
\begin{eqnarray*}
\left \|  \int _{t}^{\infty }
\frac{u(s)}{I(s)}\,\mathrm{
d}s
\right \|  _{X(0,\infty )}
& \approx & \left \|  \sum _{i=1}^{N}
a
_{i}\frac{t_{i}}{I(t_{i})}\chi _{(0,\frac{t_{i}}{2})}
\right \|  _{X(0,
\infty )}
\approx \left \|  \sum _{i=1}^{N}
a_{i}
\frac{t_{i}}{I(t_{i})} \chi _{(0,t_{i})}
\right \|  _{X(0,\infty )}
\\
&= &\|v\|_{X(0,\infty )}
\geq \left \|  \sum _{i=1}^{N}
a_{i}
\chi _{(0,t_{i})}(t)\int _{t_{i}}^{2t_{i}}\frac{\mathrm{
d}s}{I(s)}
\right \|  _{X(0,\infty )}
\\
&\approx &\left \|  \sum _{i=1}^{N}
a_{i}\chi _{(0,t_{i})}(t)\int _{0}
^{t_{i}}\frac{\mathrm{
d}s}{I(s)}
\right \|  _{X(0,\infty )}
\geq \left \|  \int _{t}^{\infty
}
\frac{u(s)}{I(s)}\,\mathrm{
d}s
\right \|  _{X(0,\infty )}.\qedhere
\end{eqnarray*}
\end{proof}

We shall also need a variant of the result obtained in
\citep[Theorem~3.9]{T2} and also~\citep[Lemma~3.3]{CP-TAMS} on the
interval $(0,\infty )$. Here we present a more general claim with a
shorter and more comprehensive proof.

In the following lemma, we work with the so-called quasiconcave
functions instead of power functions. Recall that a nonnegative function
$\varphi $ defined on $[0,\infty )$ is said to be \emph{quasiconcave}
provided that $\varphi $ is nondecreasing on $[0,\infty )$,
$\frac{\varphi (t)}{t}$ is nonincreasing on $(0,\infty )$ and
$\varphi (0)=0$. It follows that $\varphi $ is absolutely continuous
except perhaps at the origin and
%
\begin{eqnarray}
\label{E:phi_AC}
\varphi (t)-\varphi (s)\le \int _{s}^{t} \frac{\varphi (r)}{r}\,\mathrm{
d}r, \quad t\in (0,\infty), \ s\in(0,t].
\end{eqnarray}
See \citep[Chapter II, Lemma~1.1]{Kre}.

\begin{lemma}
\label{L:unsup}
Let $\varphi $ be a quasiconcave function. Then there exists a constant
$C>0$ such that
%
\begin{eqnarray}
\label{E:unsup_in}
\int _{0}^{\tau } \sup _{t\leq s<\infty } \varphi (s) f(s)\,\mathrm{
d}t
\le C \int _{0}^{\tau } (\varphi f)^{*}(t) \,\mathrm{
d}t
\end{eqnarray}
for every $\tau\in (0,\infty )$ and every nonincreasing $f\in
\mathcal{M}_{+}(0,\infty )$.

Furthermore, if $X$ is an~r.i.~space over
$(0,\infty )$, then
%
\begin{eqnarray}
\label{E:unsup_ri}
\Bigl \| \sup _{t\leq s<\infty } \varphi (s) f(s) \Bigr \|_{X(0,\infty
)}
\le C\, \bigl \| \varphi f \bigr \|_{X(0,\infty )}
\end{eqnarray}
for every nonincreasing $f\in \mathcal{M}_{+}(0,\infty )$.
\end{lemma}

\begin{proof}
Let $f\in \mathcal{M}_{+}(0,\infty )$ be a nonincreasing function and
fix $\tau \in (0,\infty )$. We split the supremum into three parts, namely
\begin{eqnarray*}
\int _{0}^{\tau } \sup _{t\le s<\infty } \varphi (s) f(s)\,\mathrm{
d}t
& \le & \int _{0}^{\tau } \sup _{t\le s\le \tau } \varphi (s) f(s)\,\mathrm{
d}t
+ \tau \, \sup _{\tau \le s<\infty } \varphi (s) f(s)
\\
& \le &\int _{0}^{\tau } \sup _{t\le s\le \tau }
\bigl [ \varphi (s) - \varphi (t)
\bigr ] f(s)\,\mathrm{
d}t
\\
&&
+ \int _{0}^{\tau } \varphi (t) \sup _{t\le s\le \tau } f(s)\,\mathrm{
d}t
+ \tau \, \sup _{\tau \le s<\infty } \varphi (s) f(s)
\\
& = &\mbox{I} + \mbox{II} + \mbox{III}.
\end{eqnarray*}
By \reftext{\eqref{E:phi_AC}} and the Hardy--Littlewood inequality, we have
\begin{eqnarray*}
\mbox{I}
& \le &\int _{0}^{\tau } \sup _{t\le s\le \tau } \biggl ( \int _{t}^{s}
\frac{
\varphi (r)}{r}\,\mathrm{
d}r \biggr ) f(s)\,\mathrm{
d}t
\le \int _{0}^{\tau } \sup _{t\le s\le \tau } \int _{t}^{s}
\frac{\varphi (r)}{r}f(r)\,\mathrm{
d}r\, \mathrm{
d}t
\\
& = &\int _{0}^{\tau } \int _{t}^{\tau } \frac{\varphi (r)}{r} f(r)\,\mathrm{
d}r\, \mathrm{
d}t
= \int _{0}^{\tau } \int _{0}^{r} \frac{\varphi (r)}{r} f(r)\,\mathrm{
d}t\, \mathrm{
d}r
\\
& = &\int _{0}^{\tau } \varphi (r) f(r)\, \mathrm{
d}r
\le \int _{0}^{\tau } (\varphi f)^{*}(t)\,\mathrm{
d}t.
\end{eqnarray*}
The second term is obviously estimated by the right hand side of \reftext{\eqref{E:unsup_in}}. Let us consider the third term. Observe that, by \reftext{\eqref{E:phi_AC}},
\begin{eqnarray*}
\varphi (2t)-\varphi (t)
\le \int _{t}^{2t} \frac{\varphi (r)}{r}\,
\mathrm{
d}r
\le \varphi (t)
, \quad t\in (0,\infty ),
\end{eqnarray*}
since $\varphi (t)/t$ is nonincreasing, whence $\varphi (2t)\le 2
\varphi (t)$ for $t\in (0,\infty )$. Using this and the fact that
$\varphi $ is nondecreasing, we get
%
\begin{eqnarray}
\label{E:phi_int}
\varphi (t)
\le 2 \varphi (t/2)
\le \frac{4}{t} \int _{t/2}^{t} \varphi
(r)\,\mathrm{
d}r
\le \frac{4}{t} \int _{0}^{t} \varphi (r)\,\mathrm{
d}r
, \quad t\in (0,\infty ).
\end{eqnarray}
Using \reftext{\eqref{E:phi_int}} we obtain
\begin{eqnarray*}
\mbox{III}
& = &\tau \, \sup _{\tau \le s<\infty } \varphi (s) f(s)
\le 4\tau
\, \sup _{\tau \le s<\infty }
\Bigl ( { \frac{1}{s} } \int _{0}^{s} \varphi
(r)\,\mathrm{
d}r \Bigr ) f(s)
\\
& \le & 4\tau \, \sup _{\tau \le s<\infty } { \frac{1}{s} } \int _{0}^{s} \varphi
(r) f(r)\,\mathrm{
d}r
\le 4\tau \, \sup _{\tau \le s<\infty } {\frac{1}{s} } \int _{0}^{s} (\varphi
f)^{*}(t)\,\mathrm{
d}t
\\
& =& 4\int _{0}^{\tau } (\varphi f)^{*}(t)\,\mathrm{
d}t,
\end{eqnarray*}
where in the second inequality we used that $f$ is nonincreasing and the
third one is due to the Hardy--Littlewood inequality. Combination
of the estimates gives \reftext{\eqref{E:unsup_in}} with $C=6$. The inequality \reftext{\eqref{E:unsup_ri}} (with the same $C$) then follows from \reftext{\eqref{E:unsup_in}} by the
Hardy-Littlewood-P\'{o}lya principle.
\end{proof}

\begin{proof}[Proof of \reftext{Theorem~\ref{T:lenka}}]
Assume first that (a) is true. Then the associate norm of the optimal
r.i.~range partner space for $X$ with respect to $M_{\gamma }$ is
equivalent to \reftext{\eqref{E:sigma-frac-corollary}} owing to
\reftext{Theorem~\ref{T:fractional-corollary}}. On the other hand, that norm is
also equivalent to \reftext{\eqref{E:sigma-frac}} by
\reftext{Theorem~\ref{T:fractional-maximal-operator}}. We recall that the
assumption~\reftext{\eqref{E:fund}} of this theorem is satisfied since it follows
from (a), as was pointed out in \reftext{Remark~\ref{R:comparison}}. Combining
these two facts, we immediately obtain (b) (see also \reftext{Remark~\ref{R:b}}).

The converse implication is considerably more involved. Suppose that (b)
holds. Then the functional
%
\begin{eqnarray}
\label{eq:plenka}
g \mapsto \left \|  \int _{t}^{\infty }
g^{*}(s)s^{{\frac{\gamma }{n}}-1}
\,\mathrm{
d}s
\right \|  _{X'(0,\infty )}
\end{eqnarray}
is equivalent to $\sigma $ from~\reftext{\eqref{E:sigma-frac}}, which in turn is
known to be an~r.i.~norm thanks to
\reftext{Theorem~\ref{T:fractional-maximal-operator}}. We note that \reftext{\eqref{E:fund}} is indeed satisfied because it follows from the proof of
\reftext{Theorem~\ref{T:fractional-maximal-operator}} that it holds if and only
if $\sigma (u) < \infty $ for every nonnegative simple function $u$,
which can be readily verified here thanks to (b). Hence the collection\looseness=1
\begin{eqnarray*}
Y(0,\infty )=\left \{  g\in \mathcal{M}(0,\infty ), \quad \left \|  \int
_{t}^{\infty }g^{*}(s)s^{{\frac{\gamma }{n}}-1}\,\mathrm{
d}s\right \|  _{X'(0,\infty )}<\infty \right \}  ,
\end{eqnarray*}
endowed with the functional
\begin{eqnarray*}
\|g\|_{Y(0,\infty )}=\left \|  \int _{t}^{\infty }g^{*}(s)s^{{\frac{
\gamma }{n}}-1}\,\mathrm{
d}s\right \|  _{X'(0,\infty )},
\end{eqnarray*}
is equivalent to an~r.i.~space. Define the operator
$R$ on $\mathcal{M}(0,\infty )$ by
\begin{eqnarray*}
Rg(t)=\int _{t}^{\infty }|g(s)|s^{{\frac{\gamma }{n}}-1}\,\mathrm{
d}s,\quad t\in (0,\infty ).
\end{eqnarray*}
Then we have
\begin{eqnarray*}
\|Rg^{*}\|_{X'(0,\infty )}=\|g\|_{Y(0,\infty )}, \quad
g\in \mathcal{M}(0,\infty ).
\end{eqnarray*}
Therefore, using also the equivalence of~\reftext{\eqref{T:Lenka-unrestricted}}
and~\reftext{\eqref{T:Lenka-nonincreasing}}, it clearly follows that
%
\begin{eqnarray}
\label{E:H-bounded}
R\colon Y(0,\infty ) \to X'(0,\infty )
\end{eqnarray}
and that $Y(0,\infty )$ is the optimal (largest possible)
r.i.~space rendering~\reftext{\eqref{E:H-bounded}} true (in
other words, it is the optimal r.i.~domain partner
space for $X'(0,\infty )$ with respect to the operator $R$).

We however claim a~considerably less obvious fact, namely that
$X'(0,\infty )$ is also the smallest possible rearrangement-invariant
space in~\reftext{\eqref{E:H-bounded}}, that is, it is the optimal
r.i.~range partner space for $Y(0,\infty )$ with
respect to the operator~$R$.

We know that $R$ is bounded from $Y(0,\infty )$ to $X'(0,\infty )$.
Therefore we are entitled to denote the optimal rearrangement-invariant
range partner for $Y(0,\infty )$ with respect to $R$ by $Y_{R}(0,
\infty )$. Denote further by $Y_{R_{D}}(0,\infty )$ the optimal
r.i.~domain partner for $Y_{R}(0,\infty )$ with
respect to $R$. Then, using the same reasoning as above, we obtain that
\begin{eqnarray*}
\|g\|_{Y_{R_{D}}(0,\infty )}\approx \left \|  \int _{t}^{\infty }g^{*}(s)s
^{{\frac{\gamma }{n}}-1}\,\mathrm{
d}s\right \|  _{Y_{R}(0,\infty )}, \quad g\in
\mathcal{M}(0,\infty ).
\end{eqnarray*}

Only an easy observation is needed to realize that once a~space is the
optimal domain partner of \textit{some} space, then it is necessarily
also the optimal domain partner to its own optimal range partner.
Indeed, knowing that $Y(0,\infty )$ is optimal in $R\colon Y(0,\infty
)\to X'(0,\infty )$, assume that $R\colon Z(0,\infty )\to Y_{R}(0,
\infty )$. By optimality of $Y_{R}(0,\infty )$ in $R\colon Y(0,\infty
)\to Y_{R}(0,\infty )$, one necessarily has $Y_{R}(0,\infty )\hookrightarrow
X'(0,\infty )$. Thus, $R\colon Z(0,\infty )\to X'(0,\infty )$. But, by
optimality of $Y(0,\infty )$ in $R\colon Y(0,\infty )\to X'(0,\infty
)$, it follows that $Z(0,\infty )\hookrightarrow Y(0,\infty )$.

Consequently, $Y(0,\infty )=Y_{R_{D}}(0,\infty )$, that is,
\begin{eqnarray*}
\left \|  \int _{t}^{\infty }g^{*}(s)s^{{\frac{\gamma }{n}}-1}\,\mathrm{
d}s\right \|  _{X'(0,\infty )}
\approx \left \|  \int _{t}^{\infty }g
^{*}(s)s^{{\frac{\gamma }{n}}-1}\,\mathrm{
d}s\right \|  _{Y_{R}(0,\infty )}, \quad g\in
\mathcal{M}(0,\infty ).
\end{eqnarray*}
Assume that $u=\sum _{i=1}^{N}b_{i}\chi _{(0,s_{i})}$ for some
$N\in \mathbb{N}$, $0 < s_{1} <\dots < s_{N}<\infty $ and $b_{1},
\dots , b_{N} > 0$. Let further
%
\begin{eqnarray}
\label{eq2:lenka}
v=\sum _{i=1}^{N} b_{i}\frac{s_{i}}{I(s_{i})}\chi _{(0,s_{i})},
\end{eqnarray}
where $I(t)=t^{1 - {\frac{\gamma }{n}}}$, $t\in (0, \infty )$. Note that
the function $I$ satisfies the assumptions of \reftext{Lemma~\ref{L:lenka}}.
Therefore we are entitled to use the lemma, whence we get
\begin{eqnarray*}
\|v\|_{X'(0,\infty )}
\approx \left \|  \int _{t}^{\infty }
u(s)s^{
{\frac{\gamma }{n}}-1}\,\mathrm{
d}s
\right \|  _{X'(0,\infty )}
\approx \left \|  \int _{t}^{\infty }
u(s)s
^{{\frac{\gamma }{n}}-1}\,\mathrm{
d}s
\right \|  _{Y_{R}(0,\infty )}
\approx \|v\|_{Y_{R}(0,\infty )}.
\end{eqnarray*}
Now, if $f\in \mathcal{M}(\mathbb{R}^{n})$, then there is a sequence
$\{v_{n}\}$ of nonnegative simple functions in the form of \reftext{\eqref{eq2:lenka}} satisfying $v_{n}\nearrow f^{*}$. By the Fatou
property and the computations above, we get $X'(0,\infty )=Y_{R}(0,
\infty )$. This proves that $X'(0,\infty )$ is indeed the optimal range
space in~\reftext{\eqref{E:H-bounded}}.

We next claim that
%
\begin{eqnarray}
\label{E:next-claim}
\|g\|_{X(0,\infty )}\approx \left \|
t^{{\frac{\gamma }{n}}}g^{**}(t)\right \|  _{Y'(0,\infty )},\quad
g\in \mathcal{M}(0,\infty ).
\end{eqnarray}
Indeed, by the definition of the associate norm, the Fubini theorem and
the H\"{o}lder inequality, one has, for every $g\in \mathcal{M}(0,
\infty )$,
\begin{eqnarray*}
\left \|  t^{{\frac{\gamma }{n}}}g^{**}(t)\right \|  _{Y'(0,\infty )}
&=&
\sup _{\|h\|_{Y(0,\infty )}\leq 1}
\int _{0}^{\infty }|h(t)|t^{\frac{
\gamma }{n}-1}\int _{0}^{t}g^{*}(s)\mathrm{
d}s\,\mathrm{
d}t
\\
&=&
\sup _{\|h\|_{Y(0,\infty )}\leq 1}
\int _{0}^{\infty }g^{*}(s)\int
_{s}^{\infty }|h(t)|t^{\frac{\gamma }{n}-1}\mathrm{
d}t\,\mathrm{
d}s
\\
&\leq &\sup _{\|h\|_{Y(0,\infty )}\leq 1}
\|g\|_{X(0,\infty )}\left \|
\int _{s}^{\infty }|h(t)|t^{\frac{\gamma }{n}-1}\,\mathrm{
d}t\right \|  _{X'(0,\infty )}.
\end{eqnarray*}
Now, the equivalence of \reftext{\eqref{T:Lenka-unrestricted}} and \reftext{\eqref{T:Lenka-nonincreasing}} implies that
\begin{eqnarray*}
\sup _{\|h\|_{Y(0,\infty )}\leq 1}\left \|  \int _{s}^{\infty }|h(t)|t
^{\frac{\gamma }{n}-1}\,\mathrm{
d}t\right \|  _{X'(0,\infty )}
\leq C
\sup _{\|h\|_{Y(0,\infty )}
\leq 1}\left \|  \int _{s}^{\infty }h^{*}(t)t^{\frac{\gamma }{n}-1}\,
\mathrm{
d}t\right \|  _{X'(0,\infty )}
= C.
\end{eqnarray*}
It might be instructive to note that while this estimate, of course,
follows from (b), the validity of (b) is in fact not necessary in order
to get it. Altogether, combining the estimates, we get
%
\begin{eqnarray}
\label{E:one-inequality}
\left \|  t^{{\frac{\gamma }{n}}}g^{**}(t)\right \|  _{Y'(0,\infty )}
\leq C \|g\|_{X(0,\infty )} , \quad g\in
\mathcal{M}(0,\infty ).
\end{eqnarray}

In order to prove~\reftext{\eqref{E:next-claim}}, we now need to show the converse
inequality to~\reftext{\eqref{E:one-inequality}}. Denote
\begin{eqnarray*}
\|g\|_{Z(0,\infty )}=\left \|
t^{{\frac{\gamma }{n}}}g^{**}(t)\right \|  _{Y'(0,\infty )},\quad g
\in \mathcal{M}(0,\infty ).
\end{eqnarray*}
The functional $g\mapsto \|g\|_{Z(0,\infty )}$ is
an~r.i.~norm. To see this, only (P4) needs proof,
since everything else is readily verified. Applying standard techniques,
(P4) reduces to
%
\begin{eqnarray}
\label{cond1:lenka}
t^{{\frac{\gamma }{n}}-1}\chi _{[1,\infty )}(t)\in Y'(0,\infty ).
\end{eqnarray}
But, using the equivalence of \reftext{\eqref{T:Lenka-unrestricted}} and \reftext{\eqref{T:Lenka-nonincreasing}} once again, we get
\begin{eqnarray*}
\|t^{{\frac{\gamma }{n}}-1}\chi _{[1,\infty )}(t)\|_{Y'(0,\infty )}
&=&
\sup \limits _{\|f\|_{Y(0,\infty )}\leq 1}\int _{0}^{\infty }|f(t)|t^{
{\frac{\gamma }{n}}-1}\chi _{[1,\infty )}(t)\,\mathrm{
d}t
\\
&=&\frac{1}{\|\chi _{(0,1)}\|_{X'(0,\infty )}}\sup
\limits _{\|f\|_{Y(0,\infty )}\leq 1}\left \|  \chi _{(0,1)}\int _{1}
^{\infty }|f(t)|t^{{\frac{\gamma }{n}}-1}\,\mathrm{
d}t\right \|  _{X'(0,\infty )}
\\
&\leq & \frac{1}{\|\chi _{(0,1)}\|_{X'(0,\infty )}}\sup
\limits _{\|f\|_{Y(0,\infty )}\leq 1}\left \|  \int _{s}^{\infty }|f(t)|t
^{{\frac{\gamma }{n}}-1}\,\mathrm{
d}t\right \|  _{X'(0,\infty )}
\\
&\le &\frac{C_2}{\|\chi _{(0,1)}\|_{X'(0,\infty )}}\sup
\limits _{\|f\|_{Y(0,\infty )}\leq 1}\left \|  \int _{s}^{\infty }f^{*}(t)t
^{{\frac{\gamma }{n}}-1}\,\mathrm{
d}t\right \|  _{X'(0,\infty )}
\\
&\le &\frac{C_2}{\|\chi _{(0,1)}\|_{X'(0,\infty )}} < \infty 
\end{eqnarray*}
for some appropriate positive constant $C_2$. We define the operator $R'$ by
\begin{eqnarray*}
R'g(t)=t^{\frac{\gamma }{n}-1}\int _{0}^{t}|g(s)|\,\mathrm{
d}s,\quad g\in \mathcal{M}(0,\infty ).
\end{eqnarray*}
Then
\begin{eqnarray*}
R'\colon Z(0,\infty )\to Y'(0,\infty ),
\end{eqnarray*}
since, by the Hardy--Littlewood inequality,
%
\begin{eqnarray}
\label{E:H-from-Y'-to-Z'}
\|R'g\|_{Y'(0,\infty )}\leq \|R'g^{*}\|_{Y'(0,\infty )}=\|g\|_{Z(0,
\infty )}, \quad g\in \mathcal{M}(0,\infty ).
\end{eqnarray}
We also have
%
\begin{eqnarray}
\label{E:H-from-Z'-to-Y}
R\colon Y(0,\infty )\to Z'(0,\infty ),
\end{eqnarray}
since, by the Fubini theorem, the H\"{o}lder inequality
and~\reftext{\eqref{E:H-from-Y'-to-Z'}}, one has
\begin{eqnarray*}
\|Rg\|_{Z'(0,\infty )}
&=&
\sup _{\|f\|_{Z(0,\infty )}\leq 1}\int _{0}
^{\infty }f(t)Rg(t)\,dt
=
\sup _{\|f\|_{Z(0,\infty )}\leq 1}\int _{0}
^{\infty }|f(t)|Rg(t)\,dt
\\
&=&
\sup _{\|f\|_{Z(0,\infty )}\leq 1}\int _{0}^{\infty }R'f(t)|g(t)|\,dt
\leq \|g\|_{Y(0,\infty )}\sup _{\|f\|_{Z(0,\infty )}\leq 1}\|R'f\|_{Y'(0,
\infty )}
\\
&\leq &\|g\|_{Y(0,\infty )}.
\end{eqnarray*}
But, as we know, $X'(0,\infty )$ is the optimal (smallest)
r.i.~target partner for $Y(0,\infty )$ with respect
to $R$. Consequently, it must be contained in $Z'(0,\infty )$.
By~\reftext{\eqref{E:equivalence-of-identities}}, this means that $Z(0,\infty )$
is continuously embedded into $X(0,\infty )$. In other words, there
exists a~positive constant, $C'$, such that
%
\begin{eqnarray}
\label{E:converse-inequality}
\|g\|_{X(0,\infty )}\leq C' \|g\|_{Z(0,\infty )}=C'\left \|  t^{{\frac{
\gamma }{n}}}g^{**}(t)\right \|  _{Y'(0,\infty )},\quad
g\in \mathcal{M}(0,\infty );
\end{eqnarray}
hence~\reftext{\eqref{E:next-claim}} follows from the combination
of~\reftext{\eqref{E:one-inequality}} and~\reftext{\eqref{E:converse-inequality}}.

Now we know that $X(0,\infty )=Z(0,\infty )$, so in order to prove (a)
it suffices to show that $T_{\frac{\gamma }{n}}\colon Z(0,\infty )
\to Z(0,\infty )$. In other words, we claim that there exists a~positive
constant $C$ such that
%
\begin{eqnarray}
\label{E:last-claim}
\left \|
t^{{\frac{\gamma }{n}}}(T_{\frac{\gamma }{n}}g)^{**}(t)\right \|
_{Y'(0,\infty )}
\leq C
\left \|
t^{{\frac{\gamma }{n}}}g^{**}(t)\right \|  _{Y'(0,\infty )},\quad
g\in \mathcal{M}(0,\infty ).
\end{eqnarray}

We first recall that there exists a~positive constant $K$ depending only
on $n$ and $\gamma $ such that
%
\begin{eqnarray}
\label{E:doublestar-in}
(T_{\frac{\gamma }{n}}g)^{**}(t)
\leq K
T_{\frac{\gamma }{n}}(g^{**})(t)
, \quad g\in \mathcal{M}(0,\infty ),\quad t
\in (0,\infty ).
\end{eqnarray}
Indeed, this follows from~\citep[Lemma~4.1]{Mus:18}, where a more
general assertion is stated and proved.

Next, it follows from \reftext{Lemma~\ref{L:unsup}} that
\begin{eqnarray*}
\left \|
\sup _{t\leq s<\infty }s^{{\frac{\gamma }{n}}}g^{*}(s)\right \|  _{Y'(0,
\infty )}
\leq C
\left \|  t^{{\frac{\gamma }{n}}}g^{*}(t)\right \|
_{Y'(0,\infty )}, \quad g\in \mathcal{M}(0,\infty
).
\end{eqnarray*}
In particular, since $g^{**}$ is also nonincreasing, we have
%
\begin{eqnarray}
\label{E:for-double}
\left \|
\sup _{t\leq s<\infty }s^{{\frac{\gamma }{n}}}g^{**}(s)\right \|  _{Y'(0,
\infty )}
\leq C
\left \|  t^{{\frac{\gamma }{n}}}g^{**}(t)\right \|
_{Y'(0,\infty )}, \quad g\in \mathcal{M}(0,\infty
).
\end{eqnarray}
Thus, combining~\reftext{\eqref{E:doublestar-in}} and~\reftext{\eqref{E:for-double}}, we get
\begin{eqnarray*}
\left \|
t^{{\frac{\gamma }{n}}}(T_{\frac{\gamma }{n}}g)^{**}(t)\right \|
_{Y'(0,\infty )}
&\leq & K
\left \|  t^{{\frac{\gamma }{n}}}T_{\frac{
\gamma }{n}}(g^{**})(t)\right \|  _{Y'(0,\infty )}
=K
\left \|
\sup _{t\leq s<\infty }s^{{\frac{\gamma }{n}}}g^{**}(s)\right \|  _{Y'(0,
\infty )}
\\
&\leq& KC
\left \|  t^{{\frac{\gamma }{n}}}g^{**}(t)\right \|  _{Y'(0,
\infty )}, \quad g\in \mathcal{M}(0,\infty ),
\end{eqnarray*}
proving~\reftext{\eqref{E:last-claim}}. Hence (a) holds, as desired. The proof is
complete.
\end{proof}

Let us now turn our attention to proofs of the main results.

\begin{proof}[Proof of \reftext{Theorem~\ref{T:fractional-maximal-operator}}]
We begin by proving that $\sigma $ is an~r.i.~norm.
As in the proof of \reftext{Theorem~\ref{T:maximal-operator}}, only the triangle
inequality and axioms (P4) and (P5) have to be verified. The triangle
inequality follows by the same argument using measure-preserving
transformations as in \citep[Theorem~3.3]{T2}.

We shall verify the validity of (P4). Let $E\subset \mathbb{R}$ be a
measurable set with $|E|<\infty $ and let $h$ be such that $h\sim \chi
_{E}$. We infer that there is a measurable set $F\subset \mathbb{R}$
such that $h=\chi _{F}$ and $|F|=|E|$. Assume moreover that $|E|\ge 1$.
It follows from the regularity of the Lebesgue measure that there exists
an open set $G \supseteq F$ such that $|G|\le 2|F|$. Thus there are
disjoint intervals $(a_{k},b_{k})$ satisfying $|F| \le a_{k}$,
\begin{eqnarray*}
F\subseteq (0,|F|) \cup \bigcup _{k} (a_{k},b_{k})
\end{eqnarray*}
and
\begin{eqnarray*}
\sum _{k} (b_{k} - a_{k}) \le 2|F|.
\end{eqnarray*}
Then we have
\begin{eqnarray*}
\biggl \| \int _{t}^{\infty } h(s)\,s^{{\frac{\gamma }{n}}-1}\,\mathrm{
d}s \biggr \|_{X'(0,\infty )}
& \le &\biggl \| \int _{t}^{\infty }
\biggl (\chi _{(0,|F|)}(s) + \sum _{k} \chi _{(a_{k},b_{k})}(s) \biggr )\,s
^{{\frac{\gamma }{n}}-1}\,\mathrm{
d}s \biggr \|_{X'(0,\infty )}
\\
& \le &\biggl \| \int _{t}^{\infty } \chi _{(0,|F|)}(s)\,s^{{\frac{
\gamma }{n}}-1}\,\mathrm{
d}s \biggr \|_{X'(0,\infty )}
\\
&&  + \sum _{k} \biggl \| \int _{t}^{\infty } \chi _{(a_{k},b_{k})}(s)
\,s^{{\frac{\gamma }{n}}-1}\,\mathrm{
d}s \biggr \|_{X'(0,\infty )}
\\
& \le &{\frac{n}{\gamma }}|F|^{\frac{\gamma }{n}}\, \| \chi _{(0,|F|)}
\|_{X'(0,\infty )}
\\
&&  + \sum _{k} \biggl \| \chi _{(0,a_{k})}(t) \int _{a_{k}}^{b_{k}}
\,s^{{\frac{\gamma }{n}}-1}\,\mathrm{
d}s \biggr \|_{X'(0,\infty )}
\\
&&  + \sum _{k} \biggl \| \chi _{(a_{k},b_{k})}(t) \int _{t}^{b_{k}}
\,s^{{\frac{\gamma }{n}}-1}\,\mathrm{
d}s \biggr \|_{X'(0,\infty )}.
\end{eqnarray*}
Let us observe that, due to~\reftext{\eqref{E:fundamental-relation}},
\reftext{\eqref{E:fund}} is in fact equivalent to the existence of a constant
$C$ such that
%
\begin{eqnarray}
\label{cokl}
r^{{\frac{\gamma }{n}}-1} \|\chi _{(0,r)}\|_{X'(0,\infty )} \le C
, \quad r\in[1,\infty).
\end{eqnarray}
Next, using the monotonicity of $s^{{\frac{\gamma }{n}}-1}$ and \reftext{\eqref{cokl}}, we get (note that $a_{k}\geq 1$ is satisfied thanks to
$a_{k}\geq |F|$)
\begin{eqnarray*}
\biggl \| \chi _{(0,a_{k})}(t) \int _{a_{k}}^{b_{k}} \,s^{{\frac{\gamma
}{n}}-1}\,\mathrm{
d}s \biggr \|_{X'(0,\infty )}
\le a_{k}^{{\frac{\gamma }{n}}-1} \|
\chi _{(0,a_{k})} \|_{X'(0,\infty )}\, (b_{k}-a_{k})
\le C (b_{k} - a
_{k}).
\end{eqnarray*}
Note that $C$ is independent of $k$. Also,
\begin{eqnarray*}
\biggl \| \chi _{(a_{k},b_{k})}(t) \int _{t}^{b_{k}} \,s^{{\frac{\gamma
}{n}}-1}\,\mathrm{
d}s \biggr \|_{X'(0,\infty )}
&\le &\biggl \| \chi _{(a_{k},b_{k})}(t)
\int _{a_{k}}^{b_{k}} \,s^{{\frac{\gamma }{n}}-1}\,\mathrm{
d}s \biggr \|_{X'(0,\infty )}
\\
& \le &a_{k}^{{\frac{\gamma }{n}}-1} \|\chi _{(0,b_{k}-a_{k})} \|_{X'(0,
\infty )}\, (b_{k}-a_{k})
\\
& \le & a_{k}^{{\frac{\gamma }{n}}-1} \|\chi _{(0,a_{k})} \|_{X'(0,
\infty )}\, (b_{k}-a_{k})
\le C (b_{k} - a_{k}),
\end{eqnarray*}
where we, once again, used the
monotonicity,~\reftext{\eqref{E:fundamental-relation}},~\reftext{\eqref{cokl}} and
\begin{eqnarray*}
b_{k}-a_{k} \le |F| \le a_{k}.
\end{eqnarray*}
Therefore
\begin{eqnarray*}
\biggl \| \int _{t}^{\infty } h(s)\,s^{{\frac{\gamma }{n}}-1}\,\mathrm{
d}s \biggr \|_{X'(0,\infty )}
& \le &{\frac{n}{\gamma }}|F|^{\frac{
\gamma }{n}}\, \| \chi _{(0,|F|)} \|_{X'(0,\infty )}
+ 2C \sum _{k} (b_{k}-a
_{k})
\\
& \le & {\frac{n}{\gamma }}C |F| + 4C |F|
= C_{n,\gamma }|E|.
\end{eqnarray*}
Taking the supremum over all such $h$, we get
%
\begin{eqnarray}
\label{vlk}
\sigma (\chi _{E}) \le C_{n,\gamma }|E|.
\end{eqnarray}
If $E\subset \mathbb{R}^{n}$ has $|E|<1$, we get $\sigma (\chi
_{E})\le C_{n,\gamma }$ by the monotonicity of $\sigma $.

As for (P5), let $E$ be a measurable subset of $\mathbb{R}^{n}$ having
finite measure

and assume that $f\in L^{1}(E)$. Denote $r=|E|$ and set $h(s) = f^{*}(s-r)
\chi _{(r,2r)}(s)$. Then $f\sim h$ and
%
\begin{eqnarray}
%
\sigma (f)
& \ge &\biggl \| \int _{t}^{\infty } h(s)\,s^{{\frac{\gamma
}{n}}-1}\,\mathrm{
d}s \biggr \|_{X'(0,\infty )}
\nonumber\\
& =& \biggl \| \int _{t}^{\infty } f^{*}(s-r)\chi _{(r,2r)}(s)\,s^{
{\frac{\gamma }{n}}-1}\,\mathrm{
d}s \biggr \|_{X'(0,\infty )}
\nonumber\\
& \ge &\biggl \| \chi _{(0,r)}(t) \int _{r}^{2r} f^{*}(s-r)\,s^{{\frac{
\gamma }{n}}-1}\,\mathrm{
d}s \biggr \|_{X'(0,\infty )} \label{E:sigma-lower-bound} \\
& =& \| \chi _{(0,r)} \|_{X'(0,\infty )}
\int _{r}^{2r} f^{*}(s-r)\,s
^{{\frac{\gamma }{n}}-1}\,\mathrm{
d}s
\nonumber\\
& \ge &\| \chi _{(0,r)} \|_{X'(0,\infty )}
(2r)^{{\frac{\gamma }{n}}-1}
\int _{r}^{2r} f^{*}(s-r)\,\mathrm{
d}s
\nonumber\\
& \ge &C_{n,\gamma ,X}
\|f\|_{L^{1}(E)},\nonumber
%
\end{eqnarray}
and (P5) follows.

We now claim that $M_{\gamma }\colon X\to Y$. Assume that $g\in
\mathcal{M}_{+}(0,\infty )$. Define $f(x)=g(\omega _{n}|x|^{n})$ for
$x\in \mathbb{R}^{n}\setminus \{0\}$, where $\omega _{n}$ is the volume
of the $n$-dimensional unit ball. Then $f$ is defined almost everywhere
on $\mathbb{R}^{n}$ and one has $g\sim f$. Thus, by the definitions of
$\sigma $ and $Y$, we get
\begin{eqnarray*}
\left \|  \int _{t}^{\infty }g(s)s^{\frac{\gamma }{n}-1}\,\mathrm{
d}s\right \|  _{X'(0,\infty )}
\leq \sigma (f)=\|f\|_{Y'}=\|g\|_{Y'(0,
\infty )}.
\end{eqnarray*}
Since $g$ was arbitrary, we obtain by~\reftext{\eqref{E:duality-of-operators}},
\begin{eqnarray*}
\left \|  t^{\frac{\gamma }{n}-1}\int _{0}^{t}g(s)\,\mathrm{
d}s\right \|  _{Y(0,\infty )}\leq \|g\|_{X(0,\infty )},\quad
g\in \mathcal{M}_{+}(0,\infty ).
\end{eqnarray*}
Restricting this inequality to nonincreasing functions, we obtain that
\begin{eqnarray*}
\left \|  t^{\frac{\gamma }{n}}g^{**}(t)\right \|  _{Y(0,\infty )}
\leq \|g^{*}\|_{X(0,\infty )}, \quad g\in
\mathcal{M}_{+}(0,\infty ).
\end{eqnarray*}
Applying \reftext{Lemma~\ref{L:unsup}}, we get that there exists a~positive
constant $C$ such that
\begin{eqnarray*}
\left \|  \sup _{t\leq s<\infty }s^{\frac{\gamma }{n}}g^{**}(s)\right \|
_{Y(0,\infty )}\leq C\|g^{*}\|_{X(0,\infty )},\quad g\in \mathcal{M}_{+}(0,\infty ).
\end{eqnarray*}
Thus, by~\reftext{\eqref{E:upper-bound-for-fractional}}, one has
\begin{eqnarray*}
\|M_{\gamma }f\|_{Y}
&\leq & C\left \|  \sup _{t\leq s<\infty }s^{\frac{
\gamma }{n}}f^{**}(s)\right \|  _{Y(0,\infty )}
\\
&\leq & C\|f^{*}\|_{X(0,\infty )}=C\|f\|_{X}, \quad f\in X,
\end{eqnarray*}
whence $M_{\gamma }\colon X\to Y$.

We shall now prove the optimality of the space $Y$ in~\reftext{\eqref{E:MG}}.
Suppose that for some r.i.~space $Z$, one has
$M_{\gamma }\colon X\to Z$. Let $g$ be a nonincreasing function in
$\mathcal{M}_{+}(0,\infty )$. Then there exists a~function $f_{0}
\in L^{1}_{\operatorname{loc}}(\mathbb{R}^{n})$ such that $f_{0}
\sim g$ and~\reftext{\eqref{E:lower-bound-for-fractional}} holds. Since
$M_{\gamma }\colon X\to Z$, we have
\begin{eqnarray*}
\|(M_{\gamma }f_{0})^{*}\|_{Z(0,\infty )}\leq C\|f_{0}^{*}\|_{X(0,
\infty )}=C\|g^{*}\|_{X(0,\infty )}.
\end{eqnarray*}
By~\reftext{\eqref{E:lower-bound-for-fractional}}, this yields
\begin{eqnarray*}
\|\sup _{t\leq s<\infty }s^{\frac{\gamma }{n}}g^{**}(s)\|_{Z(0,\infty
)}\leq C\|g^{*}\|_{X(0,\infty )}.
\end{eqnarray*}
We emphasize that $C$ does not depend on $g$. The last estimate
trivially implies
\begin{eqnarray*}
\|t^{\frac{\gamma }{n}}g^{**}(t)\|_{Z(0,\infty )}\leq C\|g^{*}\|_{X(0,
\infty )}, \quad g\in \mathcal{M}_{+}(0,\infty ).
\end{eqnarray*}
Therefore, by the Hardy--Littlewood inequality, we obtain
\begin{eqnarray*}
\|t^{\frac{\gamma }{n}}Pg(t)\|_{Z(0,\infty )}\leq C\|g\|_{X(0,\infty
)}, \quad g\in \mathcal{M}_{+}(0,\infty ).
\end{eqnarray*}
By~\reftext{\eqref{E:duality-of-operators}}, this yields
\begin{eqnarray*}
\left \|  \int _{t}^{\infty }h(s)s^{\frac{\gamma }{n}-1}\,\mathrm{
d}s\right \|  _{X'(0,\infty )}\leq C\|h\|_{Z'(0,\infty )},\quad
h\in \mathcal{M}_{+}(0,\infty ).
\end{eqnarray*}
In particular, for every $f\in \mathcal{M}_{+}(\mathbb{R}^{n})$ and
$h\in \mathcal{M}_{+}(0,\infty )$ such that $h\sim f$, one has
%
\begin{eqnarray}
\left \|  \int _{t}^{\infty }h(s)s^{\frac{\gamma }{n}-1}\,\mathrm{
d}s\right \|  _{X'(0,\infty )}
& \leq & C\|h\|_{Z'(0,\infty )}= C\|h^{*}
\|_{Z'(0,\infty )}\nonumber\\
&=& C\|f^{*}\|_{Z'(0,\infty )}=C\|f\|_{Z'}. \label{E:sigma-Z-estimate}
\end{eqnarray}
Consequently,
\begin{eqnarray*}
\sigma (f)=\sup _{\atopfrac{h\sim f}{h\ge 0}}\left \|  \int _{t}^{\infty }h(s)s^{\frac{\gamma }{n}-1}\,
\mathrm{
d}s\right \|  _{X'(0,\infty )}\leq C\|f\|_{Z'}.
\end{eqnarray*}
By the definition of $Y$, this means that $Z'\hookrightarrow Y'$, or equivalently
$Y\hookrightarrow Z$, proving the optimality of $Y$ in~\reftext{\eqref{E:MG}}.

Finally, assume that~\reftext{\eqref{E:fund}} is not true and assume that
$M_{\gamma }\colon X\to Y$ for some r.i.~space
$Y$ over $\mathbb{R}^{n}$. Then it follows from the above that
%
\begin{eqnarray}
\label{E:contradiction}
\sup _{\atopfrac{h\sim f}{h\ge 0}}\left \|  \int _{t}^{\infty }h(s)s^{\frac{\gamma }{n}-1}\,
\mathrm{
d}s\right \|  _{X'(0,\infty )}\leq C\|f\|_{Y'}, \quad
f\in Y'.
\end{eqnarray}
Take any $f\in \mathcal{M}_{+}(\mathbb{R}^{n})$ satisfying $f^{*}=
\chi _{(0,1)}$ and let $h=\chi _{(b,1+b)}$ for some fixed but arbitrary
$b\in (1,\infty )$. Then $f\sim h$ and
\begin{eqnarray*}
\biggl \| \int _{t}^{\infty } h(s)\,s^{{\frac{\gamma }{n}}-1}\,\mathrm{
d}s \biggr \|_{X'(0,\infty )}
&= &\biggl \| \int _{t}^{\infty }
\chi _{(b,1+b)}(s) \,s^{{\frac{\gamma }{n}}-1}\,\mathrm{
d}s \biggr \|_{X'(0,\infty )}
\\
& \ge& \biggl \| \chi _{(0,b)}(t) \int _{b}^{1+b} s^{{\frac{\gamma }{n}}-1}
\,\mathrm{
d}s \biggr \|_{X'(0,\infty )}
\\
&=& \| \chi _{(0,b)} \|_{X'(0,\infty )} \int _{b}^{1+b} s^{{\frac{
\gamma }{n}}-1}\,\mathrm{
d}s
\\
& \ge &\frac{b}{\varphi _{X}(b)} (1+b)^{{\frac{\gamma }{n}}-1}
\ge 2^{
{\frac{\gamma }{n}}-1}\,\frac{b^{\frac{\gamma }{n}}}{\varphi _{X}(b)}.
\end{eqnarray*}
Since~\reftext{\eqref{E:fund}} is not satisfied, there exists a~sequence
$b_{k}\to \infty $ such that
\begin{eqnarray*}
\lim _{k\to \infty }
\frac{b_{k}^{\frac{\gamma }{n}}}{\varphi _{X}(b_{k})}=\infty .
\end{eqnarray*}
This implies that
\begin{eqnarray*}
\biggl \| \int _{t}^{\infty } h(s)\,s^{{\frac{\gamma }{n}}-1}\,\mathrm{
d}s \biggr \|_{X'(0,\infty )}=\infty .
\end{eqnarray*}
Since $\|f\|_{Y'}<\infty $ by (P4) for $Y'$, this
contradicts~\reftext{\eqref{E:contradiction}}. The proof is complete.
\end{proof}

\begin{proof}[Proof of \reftext{Theorem~\ref{T:fractional-corollary}}]
We shall first prove that $\tau $ is equivalent to the functional
in~\reftext{\eqref{E:sigma-frac-corollary}}. By the definition of the associate
space, we get
\begin{eqnarray*}
\left \|  \int _{t}^{\infty }f^{*}(s)s^{\frac{\gamma }{n}-1}\,\mathrm{
d}s\right \|  _{X'(0,\infty )}
=
\sup _{\|h\|_{X(0,\infty )}\leq 1}\int
_{0}^{\infty }h(t)\int _{t}^{\infty }f^{*}(s)s^{\frac{\gamma }{n}-1}\,
\mathrm{
d}s\,\mathrm{
d}t.
\end{eqnarray*}
Since the function $t\mapsto \int _{t}^{\infty }s^{\frac{\gamma }{n}-1}f
^{*}(s)\,\mathrm{
d}s$ is obviously nonincreasing on $(0,\infty )$ regardless of $f$, we
in fact have, by the corollary of the Hardy--Littlewood inequality
(see~\reftext{\eqref{E:corollary-of-HL}}),
\begin{eqnarray*}
\left \|  \int _{t}^{\infty }f^{*}(s)s^{\frac{\gamma }{n}-1}\,\mathrm{
d}s\right \|  _{X'(0,\infty )}
=
\sup _{\|h\|_{X(0,\infty )}\leq 1}\int
_{0}^{\infty }h^{*}(t)\int _{t}^{\infty }f^{*}(s)s^{\frac{\gamma }{n}-1}
\,\mathrm{
d}s\,\mathrm{
d}t.
\end{eqnarray*}
Thus, the Fubini theorem and the definition of $P$ yield
\begin{eqnarray*}
\left \|  \int _{t}^{\infty }f^{*}(s)s^{\frac{\gamma }{n}-1}\,\mathrm{
d}s\right \|  _{X'(0,\infty )}
=
\sup _{\|h\|_{X(0,\infty )}\leq 1}\int
_{0}^{\infty }f^{*}(s)(Ph^{*})(s)s^{\frac{\gamma }{n}}\,\mathrm{
d}s.
\end{eqnarray*}
The trivial pointwise estimate $h^{*}\leq T_{\frac{\gamma }{n}}h$
implies that $(Ph^{*})(s)\leq (PT_{\frac{\gamma }{n}}h)(s)$ for every
$h$ and every $s$. Hence, we obtain that
%
\begin{eqnarray}
\label{E:tau-lower-estimate}
\left \|  \int _{t}^{\infty }f^{*}(s)s^{\frac{\gamma }{n}-1}\,\mathrm{
d}s\right \|  _{X'(0,\infty )}\leq \tau (f).
\end{eqnarray}
To prove the converse inequality, let $K$ be the operator norm of
$T_{\frac{\gamma }{n}}$ on $X(0,\infty )$. Then, by the definition of
$\tau $, the Fubini theorem, and the H\"{o}lder inequality, we have
\begin{eqnarray*}
\tau (f)
&=&
\sup _{\|h\|_{X(0,\infty )}\leq 1}\int _{0}^{\infty }f^{*}(s)(PT
_{\frac{\gamma }{n}}h)(s)s^{\frac{\gamma }{n}}\,\mathrm{
d}s
=
\sup _{\|h\|_{X(0,\infty )}\leq 1}\int _{0}^{\infty }
(T_{\frac{
\gamma }{n}}h)(t)\int _{t}^{\infty }f^{*}(s)s^{\frac{\gamma }{n}-1}\,
\mathrm{
d}s\,\mathrm{
d}t
\\
&\leq &\sup _{\|h\|_{X(0,\infty )}\leq 1}
\|T_{\frac{\gamma }{n}}h\|
_{X(0,\infty )}\left \|  \int _{t}^{\infty }f^{*}(s)s^{
\frac{\gamma }{n}-1}\,\mathrm{
d}s\right \|  _{X'(0,\infty )}.
\end{eqnarray*}
By the definition of $K$, we arrive at
%
\begin{eqnarray}
\label{E:upper-bound-for-tau}
\tau (f)
\leq K
\left \|  \int _{t}^{\infty }f^{*}(s)s^{
\frac{\gamma }{n}-1}\,\mathrm{
d}s\right \|  _{X'(0,\infty )},
\end{eqnarray}
and the desired equivalence is established.

Now we shall prove that $\tau $ is an~r.i.~norm. We
first note that the function $s\mapsto s^{\frac{\gamma }{n}}(PT_{\frac{
\gamma }{n}}h)(s)$ is always nonincreasing on $(0,\infty )$, regardless
of $h$. This follows from the easily verified fact that the expression
$s^{\frac{\gamma }{n}}(PT_{\frac{\gamma }{n}}h)(s)$ is a~constant
multiple of the integral mean over the interval $(0,s)$ of the obviously
nonincreasing function $t\mapsto \sup _{t\leq y<\infty }y^{\frac{
\gamma }{n}}h^{*}(y)$ with respect to the measure $\mathrm{
d}\mu (t)=t^{-\frac{\gamma }{n}}\,\mathrm{
d}t$. Therefore,~\reftext{\eqref{E:subadditivity-of-doublestar}} and Hardy's lemma
yield
\begin{eqnarray*}
\tau (f+g)
&=&
\sup _{\|h\|_{X(0,\infty )}\leq 1}\int _{0}^{\infty }(f+g)^{*}(s)(PT
_{\frac{\gamma }{n}}h)(s)s^{\frac{\gamma }{n}}\,\mathrm{
d}s
\\
&\leq &\sup _{\|h\|_{X(0,\infty )}\leq 1}\int _{0}^{\infty }f^{*}(s)(PT
_{\frac{\gamma }{n}}h)(s)s^{\frac{\gamma }{n}}\,\mathrm{
d}s
+
\sup _{\|h\|_{X(0,\infty )}\leq 1}\int _{0}^{\infty }g^{*}(s)(PT
_{\frac{\gamma }{n}}h)(s)s^{\frac{\gamma }{n}}\,\mathrm{
d}s
\\
&=&
\tau (f)+\tau (g).
\end{eqnarray*}
All the other properties in (P1) as well as (P2), (P3) and (P6) are
readily verified. We shall show (P4). Let $E\subset (0,\infty )$ be of
finite measure and denote $a=|E|$. By~\reftext{\eqref{E:upper-bound-for-tau}}, one
has
\begin{eqnarray*}
\tau (\chi _{E})
&\leq & K
\left \|  \int _{t}^{\infty }\chi _{E}^{*}(s)s
^{\frac{\gamma }{n}-1}\,\mathrm{
d}s\right \|  _{X'(0,\infty )}=
K\left \|  \chi _{(0,a)}(t)\int _{t}
^{a}s^{\frac{\gamma }{n}-1}\,\mathrm{
d}s\right \|  _{X'(0,\infty )}
\\
&\leq & \frac{Kn}{ \gamma }a^{\frac{\gamma }{n}}\left \|
\chi _{(0,a)}\right \|  _{X'(0,\infty )},
\end{eqnarray*}
and so
\begin{eqnarray*}
\tau (\chi _{E})
\leq \frac{Kn}{ \gamma }a^{\frac{\gamma }{n}}\left \|
\chi _{(0,a)}(t)\right \|  _{X'(0,\infty )}<\infty
\end{eqnarray*}
by the property (P4) for $X'(0,\infty )$. It remains to verify (P5). Let
$f\in \mathcal{M}(\mathbb{R}^{n})$ and let $E\subset \mathbb{R}^{n}$ be
of finite positive measure. Denote $a=|E|$. Then, by the monotonicity of
the function
$s\mapsto s^{\frac{\gamma }{n}}(PT_{\frac{\gamma }{n}}h)(s)$ on
$(0,\infty )$, we have
\begin{eqnarray*}
\tau (f)
&=&
\sup _{\|h\|_{X(0,\infty )}\leq 1}\int _{0}^{\infty }f^{*}(s)(PT
_{\frac{\gamma }{n}}h)(s)s^{\frac{\gamma }{n}}\,\mathrm{
d}s
\geq \sup _{\|h\|_{X(0,\infty )}\leq 1}\int _{0}^{a}f^{*}(s)(PT_{\frac{
\gamma }{n}}h)(s)s^{\frac{\gamma }{n}}\,\mathrm{
d}s
\\
&\geq &\sup _{\|h\|_{X(0,\infty )}\leq 1}a^{\frac{\gamma }{n}}(PT_{\frac{
\gamma }{n}}h)(a)\int _{0}^{a}f^{*}(s)\,\mathrm{
d}s.
\end{eqnarray*}
Now let us take $h_{0}=\frac{\chi _{(0,a)}}{\|\chi _{(0,a)}\|_{X(0,
\infty )}}$. Then $\|h_{0}\|_{X(0,\infty )}=1$, whence
\begin{eqnarray*}
\sup _{\|h\|_{X(0,\infty )}\leq 1}a^{\frac{\gamma }{n}}(PT_{\frac{
\gamma }{n}}h)(a)
&\geq &a^{\frac{\gamma }{n}}(PT_{\frac{\gamma }{n}}h
_{0})(a)
\\
&=&
\frac{a^{\frac{\gamma }{n}-1}}{\|\chi _{(0,a)}\|_{X(0,\infty )}}a
^{\frac{\gamma }{n}}\int _{0}^{a}s^{-\frac{\gamma }{n}}\,\mathrm{
d}s=
\frac{n}{n-\gamma }\frac{a^{\frac{\gamma }{n}}}{\|\chi _{(0,a)}\|
_{X(0,\infty )}}.
\end{eqnarray*}
Altogether,
\begin{eqnarray*}
\int _{E}f(x)\,\mathrm{
d}x\leq \int _{0}^{a}f^{*}(s)\,\mathrm{
d}s\leq \frac{n-\gamma }{n}a^{-\frac{\gamma }{n}}\|\chi _{(0,a)}\|_{X(0,
\infty )}\tau (f),
\end{eqnarray*}
and (P5) follows. We have shown that $\tau $ is an
r.i.~norm. This entitles us to take $Y=Y(\tau ')$.

We now claim that $M_{\gamma }\colon X\to Y$.
By~\reftext{\eqref{E:tau-lower-estimate}} and since $\tau (f)=\|f\|_{Y'}$ for
every $f\in \mathcal{M}_{+}(\mathbb{R}^{n})$, we have
%
\begin{eqnarray}
\label{E:sigma-Y-estimate}
\left \|  \int _{t}^{\infty }f^{*}(s)s^{\frac{\gamma }{n}-1}\,\mathrm{
d}s\right \|  _{X'(0,\infty )}\leq \|f\|_{Y'},\quad f\in \mathcal{M}_{+}(\mathbb{R}^{n}).
\end{eqnarray}
Let $g\in \mathcal{M}_{+}(0,\infty )$ be nonincreasing. We define
$f(x)=g(\omega _{n}|x|^{n})$ for $x\in \mathbb{R}^{n}\setminus \{0\}$,
where $\omega _{n}$ is the volume of the $n$-dimensional unit ball. Then
$f$ is defined almost everywhere on $\mathbb{R}^{n}$ and one has
$g\sim f$. Therefore,~\reftext{\eqref{E:sigma-Y-estimate}} implies that
\begin{eqnarray*}
\left \|  \int _{t}^{\infty }g(s)s^{\frac{\gamma }{n}-1}\,\mathrm{
d}s\right \|  _{X'(0,\infty )}
\leq \|g\|_{Y'(0,\infty )}
\end{eqnarray*}
for every nonincreasing $g\in \mathcal{M}_{+}(0,\infty )$.
Using the equivalence of \reftext{\eqref{T:Lenka-unrestricted}} and \reftext{\eqref{T:Lenka-nonincreasing}} with the (nondecreasing) function
$I(s)=s^{1-\frac{\gamma }{n}}$, $s\in (0,\infty )$, we obtain that there
exists a~positive constant $C$ such that
\begin{eqnarray*}
\left \|  \int _{t}^{\infty }g(s)s^{\frac{\gamma }{n}-1}\,\mathrm{
d}s\right \|  _{X'(0,\infty )}\leq C\|g\|_{Y'(0,\infty )},\quad
g\in \mathcal{M}_{+}(0,\infty ).
\end{eqnarray*}
By~\reftext{\eqref{E:duality-of-operators}}, this in turn gives
\begin{eqnarray*}
\left \|  t^{\frac{\gamma }{n}-1}\int _{0}^{t}g(s)\,\mathrm{
d}s\right \|  _{Y(0,\infty )}\leq C\|g\|_{X(0,\infty )},\quad
g\in \mathcal{M}_{+}(0,\infty ).
\end{eqnarray*}
Restricting this inequality to nonincreasing functions, we obtain that
\begin{eqnarray*}
\left \|  t^{\frac{\gamma }{n}}g^{**}(t)\right \|  _{Y(0,\infty )}
\leq C\|g^{*}\|_{X(0,\infty )}, \quad g\in
\mathcal{M}_{+}(0,\infty ).
\end{eqnarray*}
Applying~\reftext{Lemma~\ref{L:unsup}}, we get that there exists a~(possibly
different) positive constant $C$ such that
\begin{eqnarray*}
\left \|  \sup _{t\leq s<\infty }s^{\frac{\gamma }{n}}g^{**}(s)\right \|
_{Y(0,\infty )}\leq C\|g^{*}\|_{X(0,\infty )},\quad g\in \mathcal{M}_{+}(0,\infty ).
\end{eqnarray*}
Thus, by~\reftext{\eqref{E:upper-bound-for-fractional}}, one has
\begin{eqnarray*}
\|M_{\gamma }f\|_{Y}
&\leq C\left \|  \sup _{t\leq s<\infty }s^{\frac{
\gamma }{n}}f^{**}(s)\right \|  _{Y(0,\infty )}\leq C\|f^{*}\|_{X(0,
\infty )}=C\|f\|_{X}, \quad f\in X,
\end{eqnarray*}
whence $M_{\gamma }\colon X\to Y$.

It remains to prove the optimality of the space $Y$. Assume that
$M_{\gamma }\colon X \to Z$ for some r.i.~space
$Z$ over $\mathbb{R}^{n}$. Then~\reftext{\eqref{E:sigma-Z-estimate}} holds thanks
to the same argument as in the proof of
\reftext{Theorem~\ref{T:fractional-maximal-operator}}, that is,
\begin{eqnarray*}
\left \|  \int _{t}^{\infty }h(s)s^{\frac{\gamma }{n}-1}\,\mathrm{
d}s\right \|  _{X'(0,\infty )}
\leq C
\|f\|_{Z'},\quad h\sim f.
\end{eqnarray*}
Since $f^{*}\sim f$, this yields, in particular,
\begin{eqnarray*}
\left \|  \int _{t}^{\infty }f^{*}(s)s^{\frac{\gamma }{n}-1}\,\mathrm{
d}s\right \|  _{X'(0,\infty )}
\leq C
\|f\|_{Z'}.
\end{eqnarray*}
This estimate combined with~\reftext{\eqref{E:upper-bound-for-tau}} yields
\begin{eqnarray*}
\tau (f)\leq KC\|f\|_{Z'},\quad f\in \mathcal{M}_{+}(\mathbb{R}^{n}).
\end{eqnarray*}
As $\tau (f)=\|f\|_{Y'}$, this means that $Z'\hookrightarrow Y'$, or
$Y\hookrightarrow Z$, proving the optimality of $Y$. The proof is
complete.
\end{proof}

\begin{proof}[Proof of \reftext{Theorem~\ref{T:fractional-maximal-operator-GLZ}}]
Note that ${L^{p,q;\mathbb{A}}}$ is equivalent to
an~r.i.~space under any of the assumptions thanks to
\citep[Theorem~7.1]{OP}.

Let us first treat the cases when $T_{{\frac{\gamma }{n}}}\colon
{L^{p,q;\mathbb{A}}}(0,\infty )\to {L^{p,q;\mathbb{A}}}(0,\infty )$. To
this end we have to investigate when there exists a positive constant
$C>0$ such that
%
\begin{eqnarray}
\label{E:onestar}
\left \|
t^{-{\frac{\gamma }{n}}}\sup _{t\le s<\infty } s^{\frac{
\gamma }{n}}f^{*}(s)
\right \|  _{L^{p,q;\mathbb{A}}}\leq C
\|f\|_{L
^{p,q;\mathbb{A}}}, \quad f\in \mathcal{M}_{+}(
\mathbb{R}^{n}).
\end{eqnarray}
We first consider the case when $q=\infty $. Then~\reftext{\eqref{E:onestar}}
reads as
%
\begin{eqnarray}
\label{E:twostar}
\sup _{0<t<\infty } t^{\frac{1}{p}-{\frac{\gamma }{n}}} \ell ^{
\mathbb{A}}(t)
\sup _{t\leq s<\infty } s^{\frac{\gamma }{n}}f^{*}(s)
\leq C
\sup _{0<t<\infty } t^{\frac{1}{p}} \ell ^{\mathbb{A}}(t)f^{*}(t).
\end{eqnarray}
One has
\begin{eqnarray*}
\sup _{0<t<\infty } t^{\frac{1}{p}-{\frac{\gamma }{n}}} \ell ^{
\mathbb{A}}(t)
& \sup _{t\leq s<\infty }s^{\frac{\gamma }{n}}f^{*}(s)
=
\sup _{0<t<\infty }t^{\frac{1}{p}-{\frac{\gamma }{n}}} \ell ^{
\mathbb{A}}(t)
\sup _{t\leq s<\infty }s^{\frac{1}{p}}
\ell ^{\mathbb{A}}(s) f^{*}(s)s^{{\frac{\gamma }{n}}-\frac{1}{p}}
\ell ^{-{\mathbb{A}}}(s)
\\
& \leq \left (\sup _{0<t<\infty }
t^{\frac{1}{p}} \ell ^{\mathbb{A}}(t)
f^{*}(t) \right )
\left (\sup _{0<t<\infty }
t^{\frac{1}{p}-{\frac{
\gamma }{n}}} \ell ^{\mathbb{A}}(t)
\sup _{t\leq s<\infty } s^{{\frac{
\gamma }{n}}-\frac{1}{p}}
\ell ^{-{\mathbb{A}}}(s)\right ).
\end{eqnarray*}
Thus,~\reftext{\eqref{E:twostar}} is obviously satisfied if $s\mapsto s^{{\frac{
\gamma }{n}}-\frac{1}{p}}\ell ^{-{\mathbb{A}}}(s)$ is equivalent to
a~nonincreasing function. This happens precisely if either $p<{\frac{n}{
\gamma }}$ or $p={\frac{n}{\gamma }}$, $\alpha _{0}\leq 0$ and
$\alpha _{\infty }\geq 0$. It is easy to see that in all the remaining
cases, that is when either $p>{\frac{n}{\gamma }}$ or $p={\frac{n}{
\gamma }}$ and $\alpha _{0}> 0$, or $p={\frac{n}{\gamma }}$,
$\alpha _{0}\leq 0$ and $\alpha _{\infty }<0$, the
inequality~\reftext{\eqref{E:twostar}} is false as one can observe by plugging the
function $f^{*}=\chi _{(0,a)}$ into the inequality for $a\in (0,1)$ or
for $a\in (1,\infty )$, respectively.

Now let us consider the case when $q<\infty $. We recall that
then~\reftext{\eqref{E:onestar}} reads as
%
\begin{eqnarray}
\label{E:tristar}
\left ( \int _{0}^{\infty }
t^{-\frac{q\gamma }{n}+\frac{q}{p}-1}
\ell ^{{\mathbb{A}}q}(t)
\sup _{t\leq s<\infty } s^{\frac{q\gamma }{n}}f
^{*}(s)^{q}\,\mathrm{
d}t
\right )^{\frac{1}{q}}
\leq C
\left ( \int _{0}^{\infty }f^{*}(t)^{q}
t^{\frac{q}{p}-1} \ell ^{{\mathbb{A}}q}(t) \,\mathrm{
d}t
\right )^{\frac{1}{q}}
\end{eqnarray}
for some $C>0$ and all $f\in \mathcal{M}_{+}(\mathbb{R}^{n})$.
By~\citep[Theorem~3.2]{GOP},~\reftext{\eqref{E:tristar}} holds if and only if
there exists a constant $K$ such that, for every $\tau \in (0,\infty
)$,
%
\begin{eqnarray}
\label{E:pes}
\tau ^{{\frac{\gamma }{n}}}
\left ( \int _{0}^{\tau }
t^{-\frac{q\gamma
}{n}+\frac{q}{p} - 1}\ell ^{{\mathbb{A}}q}(t)\,\mathrm{
d}t
\right )^{\frac{1}{q}}
\le K
\left ( \int _{0}^{\tau }
t^{
\frac{q}{p}-1}\ell ^{{\mathbb{A}}q}(t)\,\mathrm{
d}t
\right )^{\frac{1}{q}}.
\end{eqnarray}
Elementary calculation shows that~\reftext{\eqref{E:pes}} holds if and only if
$1\leq p<\frac{n}{\gamma }$. Adding all conditions together we infer
that $T_{\frac{\gamma }{n}}$ is bounded on the r.i. space ${L^{p,q;
\mathbb{A}}}(0,\infty )$ if and only if one of the conditions
(\ref{E:fractional_p1}a), (\ref{E:fractional_easy}b) or
(\ref{E:fractional_infty1}c) holds.

We are thus in a~position to use \reftext{Theorem~\ref{T:fractional-corollary}}
in these cases, hence the optimal range $Y$ for the space ${L^{p,q;
\mathbb{A}}}$ with respect to $M_{\gamma }$ satisfies
\begin{eqnarray*}
\|f\|_{Y'}
= \left \|  \int _{t}^{\infty } f^{*}(s)s^{{\frac{\gamma }{n}}-1}
\,\mathrm{
d}s
\right \|  _{({L^{p,q;\mathbb{A}}})'}.
\end{eqnarray*}
Now we have by~\citep[Theorems~6.2 and~6.6]{OP} that $({L^{p,q;
\mathbb{A}}})'=L^{p',q';-{\mathbb{A}}}$, so we in fact get
\begin{eqnarray*}
\|f\|_{Y'}
= \left \|  \int _{t}^{\infty }f^{*}(s)s^{{\frac{\gamma }{n}}-1}
\,\mathrm{
d}s
\right \|  _{L^{p',q';-{\mathbb{A}}}},
\end{eqnarray*}
that is,
\begin{eqnarray*}
\|f\|_{Y'}
= \left \|
t^{\frac{1}{p'}-\frac{1}{q'}}\ell ^{-{\mathbb{A}}}(t) \int _{t}^{\infty
}f^{*}(s)s^{{\frac{\gamma }{n}}-1}\,\mathrm{
d}s
\right \|  _{L^{q'}(0,\infty )}.
\end{eqnarray*}
When $p=1$, $q=1$, $\alpha _{0}\geq 0$ and $\alpha _{\infty }\leq 0$, this
establishes the assertion in the case (\ref{E:fractional_infty1}c). In
the particular case ${\mathbb{A}}=[0,0]$ we have
\begin{eqnarray*}
\|f\|_{Y'}
= \sup _{0<t<\infty } \ell ^{-{\mathbb{A}}}(t)
\int _{t}^{
\infty }f^{*}(s)s^{{\frac{\gamma }{n}}-1}\,\mathrm{
d}s
= \int _{0}^{\infty }f^{*}(s)s^{{\frac{\gamma }{n}}-1}\,\mathrm{
d}s
= \|f\|_{L^{{\frac{n}{\gamma }},1}},
\end{eqnarray*}
hence $Y=L^{\frac{n}{n-\gamma },\infty }$. To prove the assertion, our
next step will be to simplify the expression for $\|f\|_{Y'}$ if one of
the conditions (\ref{E:fractional_p1}a) or (\ref{E:fractional_easy}b)
holds. We start with the lower bound. One has, by monotonicity of
$f^{*}$, the change of variables and elementary estimates,
\begin{eqnarray*}
\|f\|_{Y'}
& \geq &\left \|  t^{\frac{1}{p'}-\frac{1}{q'}}
\ell ^{-{\mathbb{A}}}(t)
\int _{t}^{2t}f^{*}(s)s^{{\frac{\gamma }{n}}-1}
\,\mathrm{
d}s
\right \|  _{L^{q'}(0,\infty )}
\\
& \geq &c
\left \|  t^{\frac{1}{p'}-\frac{1}{q'}+{\frac{\gamma }{n}}}
\ell ^{-{\mathbb{A}}}(t)
f^{*}(2t)
\right \|  _{L^{q'}(0,\infty )}
\\
& \geq& c'
\left \|  t^{\frac{1}{p'}-\frac{1}{q'}+{\frac{\gamma }{n}}}
\ell ^{-{\mathbb{A}}}(t)
f^{*}(t)
\right \|  _{L^{q'}(0,\infty )}
\\
& =& c' \|f\|_{L^{r',q';-{\mathbb{A}}}},
\end{eqnarray*}
where $c,c'$ are positive constants independent of $f$ and $r$ is such
that $\tfrac{1}{r'} = \tfrac{1}{p'}+{\frac{\gamma }{n}}$. We shall show
however that the converse inequality holds as well. First let $q=1$.
Then
\begin{eqnarray*}
\|f\|_{Y'}
& = &\sup _{0<t<\infty }
t^{\frac{1}{p'}}
\ell ^{-{\mathbb{A}}}(t) \int _{t}^{\infty } f^{*}(s)s^{\frac{1}{p'}+
{\frac{\gamma }{n}}}\ell ^{-{\mathbb{A}}}(s)
s^{-\frac{1}{p'}-1}
\ell ^{{\mathbb{A}}}(s)\,\mathrm{
d}s
\\
& \le &\|f\|_{L^{r',q';-{\mathbb{A}}}}
\sup _{0<t<\infty } t^{
\frac{1}{p'}} \ell ^{-{\mathbb{A}}}(t)
\int _{t}^{\infty } s^{-
\frac{1}{p'}-1} \ell ^{{\mathbb{A}}}(s)\,\mathrm{
d}s
\\
& \approx& \|f\|_{L^{r',q';-{\mathbb{A}}}}.
\end{eqnarray*}
Now assume that $1<q\leq \infty $. Then, by the classical Hardy
inequality (see e.g.~\citep{Mu}), we get that there exists a positive
constant $C$ such that
\begin{eqnarray*}[ll]
\left \|  t^{\frac{1}{p'}-\frac{1}{q'}} \ell ^{-{\mathbb{A}}}(t)
\int
_{t}^{\infty }g(s)\,\mathrm{
d}s
\right \|  _{L^{q'}(0,\infty )} \\
\qquad \leq C
\left \|  t^{\frac{1}{p'}-
\frac{1}{q'}+1} \ell ^{-{\mathbb{A}}}(t)
g(t)
\right \|  _{L^{q'}(0,
\infty )}, \quad g\in \mathcal{M}_{+}(0,\infty ).
\end{eqnarray*}
Given $f\in \mathcal{M}$, we set
$g(t)=f^{*}(t)t^{{\frac{\gamma }{n}}-1}$, $t\in (0,\infty )$, which
leads to
\begin{eqnarray*}
\|f\|_{Y'}
\leq C
\|f\|_{L^{r',q';-{\mathbb{A}}}},
\end{eqnarray*}
hence, altogether, $Y'=L^{r',q';-{\mathbb{A}}}$. Since $1<r'<\infty $,
we have, by~\citep[Theorems~6.2 and 6.6]{OP}, that
$Y=L^{r,q;{\mathbb{A}}}$, establishing the assertion.

We shall now treat the case~(\ref{E:fractional_infty2}e). The general
formula follows directly by \reftext{\eqref{E:sigma-frac}} of
\reftext{Theorem~\ref{T:fractional-maximal-operator}} and the definition of the
norm of ${L^{p,q;\mathbb{A}}}$. Note that since $T_{\frac{\gamma }{n}}$
is not bounded on ${L^{p,q;\mathbb{A}}}$ in this case, the supremum in \reftext{\eqref{E:sigma-frac}} is essential and cannot be avoided by setting
$h=f^{*}$ as follows from \reftext{Theorem~\ref{T:lenka}}.

Let us now focus on the special case when ${\mathbb{A}}=[0,0]$. We
denote the optimal partner for $L^{{\frac{n}{\gamma }},q}$ with respect
to $M_{\gamma }$ by $Y$. Our aim is to show that $Y=L^{\infty }$ or,
equivalently, that $Y'=L^{1}$. We first notice that $L^{1}$ is (up to
equivalence) the only r.i. space whose fundamental function, denoted
by $\psi $, satisfies $\psi (t)=t$. Indeed, assume that $X$ has such
a~fundamental function. Then
\begin{eqnarray*}
\|f\|_{\Lambda (X)}=\int _{0}^{\infty }f^{*}(t)\mathrm{
d}\psi (t)=\|f\|_{L^{1}}
\end{eqnarray*}
and
\begin{eqnarray*}
\|f\|_{M(X)}=\sup _{t\in (0,\infty )}\psi (t)t^{**}(t)=
\sup _{t\in (0,\infty )}\int _{0}^{t}f^{*}(s)\mathrm{
d}s=\|f\|_{L^{1}}.
\end{eqnarray*}
Consequently, by~\citep[Chapter~2, Theorem~5.13]{BS}, we have
$\Lambda (X)=X=M(X)$, hence $X=L^{1}$. Therefore, it is enough to verify
that the fundamental function of $Y'$, $\varphi $, say, satisfies
$\varphi (t)\approx t$ for $t\in (0,\infty )$. As for the proof of the
lower bound, we make use of the same calculation as in \reftext{\eqref{E:sigma-lower-bound}} with $f=\chi _{E}$ and $|E|=t$. We obtain
\begin{eqnarray*}
\|\chi _{E}\|_{Y'}
\ge C_{n,\gamma } \,t^{\frac{\gamma }{n}}\|\chi _{(0,t)}
\|_{\bigl (L^{{\frac{n}{\gamma }},q}(0,\infty )\bigr )'},
\quad t\in(0,\infty) ,
\end{eqnarray*}
which, thanks to \reftext{\eqref{E:fundamental-relation}}, can be rewritten as
\begin{eqnarray*}
\varphi (t)
\ge C_{n,\gamma } \frac{t^{1+{\frac{\gamma }{n}}}}{\|
\chi _{(0,t)}\|_{L^{{\frac{n}{\gamma }},q}(0,\infty )}},
\quad t\in(0,\infty) ,
\end{eqnarray*}
and the estimate then follows since the fundamental function of
$L^{{\frac{n}{\gamma }},q}$ is $t^{\frac{\gamma }{n}}$. To prove the
converse inequality, let us use the same upper bound which appears in
the proof of the validity of (P4) in the proof of
\reftext{Theorem~\ref{T:fractional-maximal-operator}}. Observe that \reftext{\eqref{cokl}} now holds on the whole of $(0,\infty )$ and hence we get \reftext{\eqref{vlk}} also for all sets $E$ with $|E|<1$. That gives the desired
relation $\varphi (t)\le C_{n,\gamma} t$, $t\in (0,\infty )$.
\end{proof}

\section{The Hilbert transform} \label{sec5}

\noindent
A very important example of a~singular integral with odd kernel is the
\textit{Hilbert transform}, defined for
appropriate functions on $\mathbb{R}$ by
\begin{eqnarray*}
Hf(x) = \lim _{\varepsilon \rightarrow 0_{+}}\frac{1}{\pi }
\int _{|x-t|\geq \varepsilon }\frac{f(t)}{x-t}\,\mathrm{
d}t.
\end{eqnarray*}
This operator is defined for every function $f\colon \mathbb{R}\to
\mathbb{R}$ for which the integral converges almost everywhere. The
Hilbert transform arises in the study of boundary values of the real and
imaginary parts of analytic functions. It is a cornerstone of several
important disciplines including real and complex analysis and the theory
of PDEs. In this section we shall study its sharp boundedness properties
on r.i.~spaces over $\mathbb{R}$. A key technical
background tool will be the \textit{Stieltjes transform}, $S$, which is
defined for every nonnegative measurable function $f$ on $(0,\infty )$
by
\begin{eqnarray*}
(Sf)(t)
=
\frac{1}{t}\int _{0}^{t}f(s)\,\mathrm{
d}s+\int _{t}^{\infty }f(s)\frac{\mathrm{
d}s}{s}, \quad t\in (0,\infty ).
\end{eqnarray*}
It might be useful to note that
%
\begin{eqnarray}
\label{E:comparison-of-S-P-and-Q}
S=P+Q=P\circ Q=Q\circ P.
\end{eqnarray}

Whenever we say that the Hilbert transform is bounded from a
function space $X$ to a function space $Y$, we implicitly assume that
$H$ is well defined for every $f\in X$, that is, $f\in L^{1}_{
\operatorname{loc}}(\mathbb{R})$ and the limit in the definition of
$Hf$ exists for a.e.~$x\in \mathbb{R}$. Let us recall that,
by~\citep[Chapter~3, Theorem~4.8]{BS}, a~sufficient condition for the
existence of this limit, for a given $\mathcal{M}(
\mathbb{R})$, is
%
\begin{eqnarray}
\label{E:2.2}
(Sf^{*})(1) < \infty .
\end{eqnarray}

Our main result in this section reads as follows.

\begin{theorem}
\label{T:hilbert-transforms}
Let $X$ be an~r.i.~space over $\mathbb{R}$ such that
%
\begin{eqnarray}
\label{E:eta-satisfied}
\eta \in X'(0,\infty ),
\end{eqnarray}
where
%
\begin{eqnarray}
\label{E:definition-of-w}
\eta (t)=\chi _{(0,1]}(t)(1-\log t)+\chi _{(1,\infty )}(t)\frac{1}{t},
\ t\in (0,\infty ).
\end{eqnarray}
Define the functional $\sigma $ by
\begin{eqnarray*}
\sigma (f)=\left \|  Sf^{*}\right \|  _{X'(0,\infty )}, \quad f\in
\mathcal{M}_{+}(\mathbb{R}).
\end{eqnarray*}
Then $\sigma $ is an~r.i.~norm and
%
\begin{eqnarray}
\label{E:boundedness-hilbert}
H\colon X\to Y,
\end{eqnarray}
where $Y=Y(\sigma ')$. Moreover, $Y$ is the optimal
\textup{(}smallest\textup{)} r.i.~space for
which~\reftext{\eqref{E:boundedness-hilbert}} holds.

Conversely, if~\reftext{\eqref{E:eta-satisfied}} is not true, then there does not
exist an~r.i.~space $Y$ for
which~\reftext{\eqref{E:boundedness-hilbert}} holds.
\end{theorem}

For the optimal domain, we have the following result. Again, the proof
is analogous to the appropriate proofs above, and therefore omitted.

\begin{theorem}
Let $Y$ be an~r.i.~space over $\mathbb{R}$ such that
%
\begin{eqnarray}
\label{E:eta-condition-domain}
\eta \in Y(0,\infty ),
\end{eqnarray}
where $\eta $ is the function from~\reftext{\eqref{E:definition-of-w}}. Define the
functional $\sigma $ by
\begin{eqnarray*}
\sigma (f)=\left \|  Sf^{*}\right \|  _{Y(0,\infty )}, \quad f\in
\mathcal{M}_{+}(\mathbb{R}).
\end{eqnarray*}
Then $\sigma $ is an~r.i.~norm and
%
\begin{eqnarray}
\label{E:hilbert-bounded-domain}
H\colon X\to Y,
\end{eqnarray}
where $X=X(\sigma )$. Moreover, $X$ is the optimal
\textup{(}biggest\textup{)} r.i.~space for
which~\reftext{\eqref{E:hilbert-bounded-domain}} holds.

Conversely, if~\reftext{\eqref{E:eta-condition-domain}} is not true, then there
does not exist an~r.i.~space $X$ for
which~\reftext{\eqref{E:hilbert-bounded-domain}} holds.
\end{theorem}

We provide several examples of the optimal range partners for
Lorentz-Zygmund spaces with respect to the Hilbert transform. The proof
is similar to that of \reftext{Theorem~\ref{T:-maximal-operator-GLZ}} and
therefore omitted.

\begin{theorem}
Assume that $p,q\in [1,\infty ]$, ${\mathbb{A}}\in \mathbb{R}^{2}$. Then
\begin{eqnarray*}
H\colon L^{p,q; {\mathbb{A}}}\to
\left\{
\begin{array}{l@{\quad }l}
L^{1,1; {\mathbb{A}}-1},
&p=1, q=1, \alpha _{0} \geq 1, \alpha _{
\infty }<0, \\
L^{p,q; {\mathbb{A}}},
&1<p<\infty , \\
Y,
& p=\infty , q=1, \alpha _{0} < -1, \alpha _{\infty }\ge 0~ or \\
&p = \infty , 1 < q <\infty , \alpha _{0} + \frac{1}{q} < 0,
\alpha _{\infty }+ \frac{1}{q'} > 0,\\
L^{\infty ,\infty ; {\mathbb{A}}-1},
&p=\infty , q=\infty , \quad \alpha
_{0} \leq 0, \alpha _{\infty }>1,
\end{array}\right.
\end{eqnarray*}
where $Y$ is defined by its associate space $Y'$ whose norm is given by
\begin{eqnarray*}
\|f\|_{Y'}=\left \|  \int _{t}^{\infty }f^{**}(s)\frac{\mathrm{
d}s}{s}\right \|  _{L^{(1,q';-{\mathbb{A}}-1)}}, \quad f\in \mathcal{M}
_{+}(\mathbb{R}).
\end{eqnarray*}
These spaces are the optimal range partners with respect to~$H$.
\end{theorem}

At the end of this section, we aim to prove
\reftext{Theorem~\ref{T:hilbert-transforms}}. We start with a lemma which recalls
a well-known fact. We insert a short proof for the sake of completeness.

\begin{lemma}
\label{L:comparison-of-hilbert-and-stieltjes}
Let $X$ and $Y$ be r.i.~Banach function spaces over
$\mathbb{R}$. Assume that \reftext{\eqref{E:2.2}} is satisfied for every
$f\in X$. Then the Hilbert transform $H$ is bounded from $X$ to $Y$ if
and only if the Stieltjes transform $S$ is bounded from $X(0,\infty )$
to $Y(0,\infty )$.
\end{lemma}

\begin{proof}
Assume first that $H$ is bounded from $X$ to $Y$. Fix a~function
$f\in \mathcal{M}_{+}(0,\infty )$ such that~$(Sf^{*})(1)<\infty $. Then,
by a~simple modification of~\citep[Chapter~3, Proposition 4.10]{BS},
there exists a~function $g\in \mathcal{M}_{+}(\mathbb{R})$,
equimeasurable with $f$, such that
\begin{eqnarray*}
(Sf^{*})(t)\leq 2\pi \left (Hg\right )^{*}(t), \quad
t\in (0,\infty ).
\end{eqnarray*}
Thus, by the property (P2) of $Y$, we have
\begin{eqnarray*}
\|(Sf^{*})\|_{Y(0,\infty )}\leq 2\pi \|(Hg)^{*}\|_{Y(0,\infty )}.
\end{eqnarray*}
By the rearrangement invariance of $Y$, this turns into
\begin{eqnarray*}
\|(Sf^{*})\|_{Y(0,\infty )}\leq 2\pi \|Hg\|_{Y}.
\end{eqnarray*}
It follows from the boundedness of $H$ from $X$ to $Y$ that
\begin{eqnarray*}
\|Hg\|_{Y}\leq C\|g\|_{X}
\end{eqnarray*}
for some constant $C$, $0<C<\infty $, independent of $g$ (hence of
$f$). We thus get, altogether, using also the definition of the
representation space and the equimeasurability of $f$ and $g$, that
\begin{eqnarray*}
\|(Sf^{*})\|_{Y(0,\infty )}\leq 2C\pi \|g\|_{X}=2C\pi \|g^{*}\|_{X(0,
\infty )}=2C\pi \|f^{*}\|_{X(0,\infty )}.
\end{eqnarray*}
In other words, $S$ is bounded from $X(0,\infty )$ to $Y(0,\infty )$.

Conversely, assume that the Stieltjes transform is bounded from
$X(0,\infty )$ to $Y(0,\infty )$. By an appropriate modification
of~\citep[Chapter~3, Theorem~4.8]{BS}, there exists a positive constant
$C$ independent of $f$ such that
\begin{eqnarray*}
\left (Hf\right )^{*}(t)\leq C (Sf^{*})(t) , \quad t
\in (0,\infty ).
\end{eqnarray*}
We then get, similarly as above,
\begin{eqnarray*}
\|Hf\|_{Y}=\|(Hf)^{*}\|_{Y(0,\infty )}\leq C\|(Sf^{*})\|_{X(0,\infty
)}\leq C'\|f^{*}\|_{X(0,\infty )}=C'\|f\|_{X}
\end{eqnarray*}
for some suitable constant $C'$, proving that $H\colon X\to Y$. The
proof is complete.
\end{proof}

Our next step will be a characterization of the optimal range partner
with respect to the Stieltjes transform.

\begin{theorem}
\label{T:stieltjes-transform}
Let $X$ be an~r.i.~Banach function space over
$(0,\infty )$ such that
%
\begin{eqnarray}
\label{E:eta-satisfied2}
\eta \in X'(0,\infty ),
\end{eqnarray}
where $\eta $ is the function from~\reftext{\eqref{E:definition-of-w}}. Define the
functional $\sigma $ by
\begin{eqnarray*}
\sigma (f)=\left \|  Sf^{*}\right \|  _{X'(0,\infty )}, \quad f\in
\mathcal{M}_{+}(0,\infty ).
\end{eqnarray*}
Then $\sigma $ is an~r.i.~norm and
%
\begin{eqnarray}
\label{E:boundedness-stieltjes}
S\colon X\to Y,
\end{eqnarray}
where $Y=Y(\sigma ')$. Moreover, $Y$ is the optimal
\textup{(}smallest\textup{)} r.i.~space for
which~\reftext{\eqref{E:boundedness-stieltjes}} holds.

Conversely, if~\reftext{\eqref{E:eta-satisfied2}} is not true, then there does not
exist an~r.i.~space $Y$ for
which~\reftext{\eqref{E:boundedness-stieltjes}} holds.
\end{theorem}

\begin{proof}
Consider the functional $\sigma (f)=\left \|  Sf^{*}\right \|  _{X'(0,
\infty )}$, $ f\in \mathcal{M}_{+}(0,\infty )$. We shall prove that
$\sigma $ is an~r.i.~norm. As in the proof of
\reftext{Theorem~\ref{T:maximal-operator}}, the axioms (P2), (P3) and (P6) for
$\sigma $ are clearly satisfied. The verification of the triangle
inequality is even easier than in the proof of
\reftext{Theorem~\ref{T:maximal-operator}}. It follows
from~\reftext{\eqref{E:comparison-of-S-P-and-Q}} that
%
\begin{eqnarray}
\label{E:comparison-of-S-and-Q}
Sf^{*}=Qf^{**}, \quad f\in \mathcal{M}(0,\infty ),
\end{eqnarray}
which in conjunction with~\reftext{\eqref{E:subadditivity-of-doublestar-a}}
immediately yields the triangle inequality for~$\sigma $. As usual, all
other properties in (P1) are readily verified. Also the verification of
(P5) is easy. In fact, it immediately follows from the analogous
property of the functional $\sigma $ from
\reftext{Theorem~\ref{T:maximal-operator}}, because,
by~\reftext{\eqref{E:comparison-of-S-and-Q}}, one has $Sf^{*}\geq Qf^{*}$. It only
remains to verify the validity of (P4). To this end, let $E\subset
\mathbb{R}$ be a set of finite measure. We need to prove that
$\left \|  S\chi _{E}^{*}\right \|  _{X'}<\infty $. Calculation shows
that this is equivalent to saying that $\eta \in X'$, a fact guaranteed
by the assumption. This shows (P4), and, consequently, it completes the
proof of the fact that $\sigma $ is an~r.i.~Banach
function norm.

We shall now prove that $S\colon X\to Y$. The operator $S$ is
self-adjoint with respect to the $L^{1}$-pairing in the sense that
\begin{eqnarray*}
\int _{0}^{\infty }(Sf)(t)g(t)\,\mathrm{
d}t=\int _{0}^{\infty }f(t)(Sg)(t)\,\mathrm{
d}t
\end{eqnarray*}
for every admissible $f$ and $g$. Hence, it suffices to prove that
$S\colon Y'\to X'$. That, however, follows trivially from the definition
of $Y'$.

The proof of optimality of the space $Y$ as well as that of the
nonexistence of an~r.i.~range partner for $X$ in case
$\eta \notin X'$ is completely analogous to its counterpart from the
proof of \reftext{Theorem~\ref{T:maximal-operator}} and hence is omitted.
\end{proof}

Finally, \reftext{Theorem~\ref{T:hilbert-transforms}} immediately follows from
\reftext{Theorem~\ref{T:stieltjes-transform}} and
\reftext{Lemma~\ref{L:comparison-of-hilbert-and-stieltjes}}.

\section{The Riesz potential} \label{sec6}

%
\begin{definition}
Let $0<\gamma <n$. Then the \textit{Riesz potential} of order
$\gamma $, $I_{\gamma }$, of a~measurable function $f$ on $\mathbb{R}
^{n}$ is defined by
\begin{eqnarray*}
(I_{\gamma }f)(x)=\int _{\mathbb{R}^{n}}f(y)\phi (x-y)\,\mathrm{
d}y, \quad x\in \mathbb{R}^{n},
\end{eqnarray*}
where
\begin{eqnarray*}
\phi (y)=c(\gamma )|y|^{\gamma -n},
\quad
c(\gamma )=\Gamma \left (\frac{n-\gamma }{2}\right )\left (\pi ^{
\frac{n}{2}}2^{\gamma }\Gamma \left (\frac{\gamma }{2}\right )\right )
^{-1}.
\end{eqnarray*}
\end{definition}

We are going to make use of a~special case of the
\textit{O'Neil inequality}. In its general form \citep[Lemma~1.5]{ON},
it states that, for the convolution of two measurable functions $f,g$
on $\mathbb{R}^{n}$, defined by
\begin{eqnarray*}
(f*g)(x)=\int _{\mathbb{R}^{n}}f(x-y)g(y)\,\mathrm{
d}y, \quad x\in \mathbb{R}^{n},
\end{eqnarray*}
we have
\begin{eqnarray*}
(f*g)^{**}(t)\leq tf^{**}(t)+\int _{t}^{\infty }f^{*}(s)g^{*}(s)\,
\mathrm{
d}s , \quad t\in (0,\infty ).
\end{eqnarray*}
With the particular choice
\begin{eqnarray*}
g(x)=|x|^{\gamma -n}, \quad x\in \mathbb{R}^{n},
\end{eqnarray*}
we obtain that
\begin{eqnarray*}
(I_{\gamma }f)^{*}(t)\leq C\int _{t}^{\infty }f^{**}(s)s^{\frac{\gamma
}{n}-1}\,\mathrm{
d}s , \quad t\in (0,\infty ),
\end{eqnarray*}
with some positive constant $C$, depending on $\gamma $ and $n$, but
independent of $f$ and $t$.

This inequality is known to be sharp, but merely in a broader sense
than, for example, the corresponding estimate for the Hardy--Littlewood
maximal operator. This was firstly observed by O'Neil in the final
remark of the paper~\citep{ON}, where it is pointed out that the
inequality can be reversed when $f,g$ are radially decreasing positive
functions. Furthermore, by an appropriately modified argument
from~\citep[Theorem~10.2(iii)]{EOP}), we get that, for every
$f\in \mathcal{M}(\mathbb{R}^{n})$, there exists a~function $g\in
\mathcal{M}(0,\infty )$ equimeasurable with $f$ such that
\begin{eqnarray*}
(I_{\gamma }g)^{*}(t)\geq c\int _{t}^{\infty }f^{**}(s)s^{\frac{\gamma
}{n}-1}\,\mathrm{
d}s , \quad t\in (0,\infty ),
\end{eqnarray*}
with some constant $c$, $0<c<\infty $, depending on $\gamma $ and
$n$, but independent of $f$ and $t$.

We shall now turn our attention to a~weighted version of the Stieltjes
transform, which plays a key role in the matter of optimal spaces for
the Riesz potential.

\begin{definition}
Let $\alpha \in (1,\infty )$. The
\textit{weighted Stieltjes transform}, $S_{\alpha }$, is defined for
every nonnegative measurable function $f$ on $(0,\infty )$ by
\begin{eqnarray*}
(S_{\alpha }f)(t)
=
t^{\frac{1}{\alpha }-1}\int _{0}^{t}f(s)\,\mathrm{
d}s+\int _{t}^{\infty }f(s)s^{\frac{1}{\alpha }-1}\, \mathrm{
d}s, \quad t\in (0,\infty ).
\end{eqnarray*}
\end{definition}

We note that, for every admissible $f$ and $t$, one has
\begin{eqnarray*}
(S_{\alpha }f)(t)=c_{\alpha }\int _{t}^{\infty }(Pf)(s)s^{\frac{1}{
\alpha }-1}\, \mathrm{
d}s,
\end{eqnarray*}
where $c_{\alpha }=\frac{\alpha -1}{\alpha }$.

Our main result of this section reads as follows.
%
\begin{theorem}
\label{T:riesz-potential}
Let $\gamma \in (0, n)$ and let $X$ be an~r.i.~space
over $\mathbb{R}^{n}$ such that
%
\begin{eqnarray}
\label{E:xi-satisfied}
\xi _{\frac{n}{\gamma }} \in X'(0,\infty ),
\end{eqnarray}
where, for $\alpha >0$,
%
\begin{eqnarray}
\label{E:definition-of-xi}
\xi _{\alpha }(t)=(t+1)^{\frac{1}{\alpha }-1}, \quad t\in (0,\infty ).
\end{eqnarray}
Define the functional $\sigma $ by
\begin{eqnarray*}
\sigma (f)=\left \|  S_{\frac{n}{\gamma }}f^{*}\right \|  _{X'(0,
\infty )}, \quad f\in \mathcal{M}_{+}(\mathbb{R}^{n}).
\end{eqnarray*}
Then $\sigma $ is an~r.i.~norm and
%
\begin{eqnarray}
\label{E:boundedness-riesz}
I_{\gamma }\colon X\to Y,
\end{eqnarray}
where $Y=Y(\sigma ')$. Moreover, $Y$ is the optimal
\textup{(}smallest\textup{)} r.i.~space for
which~\reftext{\eqref{E:boundedness-riesz}} holds.

Conversely, if~\reftext{\eqref{E:xi-satisfied}} is not true, then there does not
exist an~r.i.~space $Y$ for
which~\reftext{\eqref{E:boundedness-riesz}} holds.
\end{theorem}

As in the preceding sections, we also characterize optimal domains. We
also omit the proof since it is analogous, again, to that of
\reftext{Theorem~\ref{T:maximal-operator-domain}}.

\begin{theorem}
Let $\gamma \in (0, n)$ and let $Y$ be an~r.i.~space
over $\mathbb{R}^{n}$ such that
%
\begin{eqnarray}
\label{E:xi-condition-domain}
\xi _{\frac{\gamma }{n}} \in Y(0,\infty ),
\end{eqnarray}
where $\xi _{\alpha }$ is the function from~\reftext{\eqref{E:definition-of-xi}}.
Define the functional $\sigma $ by
\begin{eqnarray*}
\sigma (f)=\left \|  S_{\frac{n}{\gamma }}f^{*}\right \|  _{Y(0,
\infty )}, \quad f\in \mathcal{M}_{+}(\mathbb{R}^{n}).
\end{eqnarray*}
Then $\sigma $ is an~r.i.~norm and
%
\begin{eqnarray}
\label{E:riesz-bounded-domain}
I_{\gamma }\colon X\to Y,
\end{eqnarray}
where $X=X(\sigma )$. Moreover, $X$ is the optimal
\textup{(}biggest\textup{)} r.i.~space for
which~\reftext{\eqref{E:riesz-bounded-domain}} holds.

Conversely, if~\reftext{\eqref{E:xi-condition-domain}} is not true, then there
does not exist an~r.i.~space $X$ for
which~\reftext{\eqref{E:riesz-bounded-domain}} holds.
\end{theorem}

We use \reftext{Theorem~\ref{T:riesz-potential}} to provide several examples of
the optimal range partners for Lorentz-Zygmund spaces with respect to
the Riesz potential.

\begin{theorem}
\label{T:riesz-potential-GLZ}
Assume that $\gamma \in (0,n)$, $p,q\in [1,\infty ]$, ${\mathbb{A}}
\in \mathbb{R}^{2}$. Then
\begin{eqnarray}
\label{E:riesz_p1}
I_{\gamma }\colon {L^{p,q;\mathbb{A}}}\to
\left\{
\begin{array}{l@{\quad }l}
Y_{1}
& p=1, q=1, \alpha _{0} \geq 0, \alpha _{\infty }\leq 0, \\
L^{\frac{np}{n-\gamma p},q;{\mathbb{A}}}
& 1<p<\frac{n}{\gamma }, \\
L^{\infty ,q;{\mathbb{A}}- 1}
& p={\frac{n}{\gamma }}, 1\leq q
\leq \infty , \alpha _{0} < \frac{1}{q'},
\alpha _{\infty }>
\frac{1}{q'},
\\
L^{\infty ,q;[-\frac{1}{q},\alpha _{\infty }- 1],\left [-1, 0\right ]}
& p={\frac{n}{\gamma }}, 1< q\leq \infty , \alpha _{0} =
\frac{1}{q'},
\alpha _{\infty }>\frac{1}{q'},
\\
Y_{2}
& p={\frac{n}{\gamma }}, q=1, \alpha _{0}<0, \alpha _{
\infty }=0,
\\
L^{\infty ,1;\left [-1,\alpha _{\infty }- 1\right ],\left [-1,0\right ],
\left [-1,0\right ]}
& p={\frac{n}{\gamma }}, q=1, \alpha _{0} = 0,
\alpha _{\infty }> 0,
\\
L^{\infty }
& p={\frac{n}{\gamma }}, q=1, \alpha _{0}\geq 0,
\alpha _{\infty }= 0,
\\
Y_{3}
& p={\frac{n}{\gamma }}, q=1, \alpha _{0} > 0,
\alpha _{\infty }> 0,
\\
Y_{2}
& p={\frac{n}{\gamma }}, 1< q\leq \infty , \alpha _{0} >
\frac{1}{q'},
\alpha _{\infty }>\frac{1}{q'}
,
\end{array}\right.
\end{eqnarray}
where
\begin{eqnarray*}
\|f\|_{Y_{2}}
& = & \|f\|_{L^{\infty }}
+ \| t^{-\frac{1}{q}}
\ell ^{\alpha _{\infty }- 1}(t)f^{*}(t)\|_{L^{q}(1,\infty )},
\\
\|f\|_{Y_{3}}
& = & \| t^{-1}\ell ^{\alpha _{0} - 1}(t)f^{*}(t)\|_{{L
^{1}(0,1)}},
\end{eqnarray*}
and $Y_{1}$ is defined by its associate space $Y_{1}'$ whose norm is
given by
\begin{eqnarray*}
\|f\|_{Y_{1}'}
= \sup _{0<t<\infty } \ell ^{-{\mathbb{A}}}(t)
\int _{t}
^{\infty }f^{**}(s)s^{{\frac{\gamma }{n}}-1}\,\mathrm{
d}s, \quad f\in \mathcal{M}_{+}(\mathbb{R}^{n}).
\end{eqnarray*}
In particular, if ${\mathbb{A}}=[0,0]$, we have $Y_{1}=L^{\frac{n}{n-
\gamma },\infty }$.

Moreover, these spaces are the optimal range partners with respect to~$I
_{\gamma }$.
\end{theorem}

\begin{proof}
We note that ${L^{p,q;\mathbb{A}}}$ is equivalent to a
rearrangement--invariant Banach function space due to
\citep[Theorem~7.1]{OP} in all the cases.

Assume that $p\in (1,\infty )$ and $q\in [1,\infty ]$. By
\citep[Theorems~6.2 and~6.6]{OP}, the associate space of $L^{p,q;{\mathbb{A}}}$
is equivalent to $L^{p',q';-{\mathbb{A}}}$. We need to check
when $\xi _{\frac{n}{\gamma }}\in X'(0,\infty )$ is satisfied, that is,
when
\begin{equation*}
\int_0^\infty t^{\frac{q'}{p'} - 1}\ell^{-\mathbb{A} q'}(t)(t + 1)^{\frac{\gamma - n}{n}q'}\, \mathrm{d} t < \infty\quad\text{if $q\in(1,\infty]$,}
\end{equation*}
or when
\begin{equation*}
\sup\limits_{t\in(0,\infty)} t^{\frac1{p'}}\ell^{-\mathbb{A}}(t)(t + 1)^{\frac{\gamma - n}{n}} < \infty\quad\text{if $q = 1$.}
\end{equation*}
It is easy to see that in the former case the integral is finite if and only if either
\begin{equation*}
p\in(1,\tfrac{n}{\gamma})
\end{equation*}
or
\begin{equation*}
p=\frac{n}{\gamma}\text{ and }\alpha_\infty > \frac{1}{q'},
\end{equation*}
while in the latter case the supremum is finite if and only if either
\begin{equation*}
p\in(1,\tfrac{n}{\gamma})
\end{equation*}
or
\begin{equation*}
p=\frac{n}{\gamma}\text{ and }\alpha_\infty \geq 0.
\end{equation*}
Henceforth, we assume that these conditions are satisfied.  By the classical weighted
Hardy inequality, there is a positive constant $C$ such that
\begin{eqnarray*}
\Vert S_{\frac{n}{\gamma }} g^{*}\Vert _{p',q';-{\mathbb{A}}}
&=&
\Vert t^{\frac{1}{p'} - \frac{1}{q'}}\ell ^{-A}(t)\int _{t}^{\infty }g
^{**}(s)s^{\frac{\gamma -n}{n}}\, \mathrm{
d}s\Vert _{q'}
\\
&\le   & C\Vert t^{\frac{1}{p'}+\frac{1}{q}}\ell ^{-{\mathbb{A}}}(t)g
^{**}(t)t^{\frac{\gamma }{n} - 1}\Vert _{q'} = C\Vert t^{\frac{1}{p'}+\frac{
\gamma }{n} - \frac{1}{q'}}\ell ^{-{\mathbb{A}}}(t)g^{**}(t)\Vert _{q'}
\\
&= &C\Vert g\Vert _{(r',q';-{\mathbb{A}})},
\end{eqnarray*}
where $\frac{1}{p'} + \frac{\gamma }{n} = \frac{1}{r'}$, that is,
$r'=\frac{np}{(n+\gamma )p - n}$. The converse inequality
\begin{eqnarray*}
\Vert g\Vert _{(r',q';-{\mathbb{A}})}\le C' \Vert S_{\frac{n}{
\gamma }} g^{*}\Vert _{p',q';-{\mathbb{A}}}
\end{eqnarray*}
for some positive $C'$ follows immediately from the estimate
\begin{eqnarray*}
S_{\frac{n}{\gamma }} g^{*}(t)
&= \int _{t}^{\infty }g^{**}(s)s^{\frac{
\gamma }{n}-1}\, \mathrm{
d}s = \int _{t}^{\infty }\frac{1}{s^{2-\frac{\gamma }{n}}}\int _{0}^{s}
g^{*}(u)\, \mathrm{
d}u\, \mathrm{
d}s
\\
&\geq \int _{0}^{t} g^{*}(u)\, \mathrm{
d}u \int _{t}^{\infty }\frac{1}{s^{2-\frac{\gamma }{n}}}\, \mathrm{
d}s = \frac{n}{n-\gamma }t^{\frac{\gamma }{n}}g^{**}(t).
\end{eqnarray*}
If $p\in (1,\frac{n}{\gamma })$, then $r'\in (1,\frac{n}{\gamma })$.
By~\citep[Theorem~3.8]{OP}, $L^{(r',q';-{\mathbb{A}})}$ is equivalent
to $L^{r',q';-{\mathbb{A}}}$. Hence $Y$ is equivalent to $L^{r,q;
{\mathbb{A}}}$, where $r=\frac{np}{n-\gamma p}\in (\frac{n}{n-\gamma
},\infty )$, by \citep[Theorem~6.2]{OP}.
If $p = \frac{n}{\gamma }$, then $r' = 1$. If $q\in (1,\infty )$ (and
hence $q'\in (1,\infty )$), we obtain \reftext{\eqref{E:riesz_p1}} for
$q\in (1,\infty )$ by virtue of \citep[Theorem~6.7]{OP}. If
$q=\infty $ (and hence $q'=1$), we combine \citep[Theorem~3.8]{OP} with
\citep[Theorem~6.6]{OP} in order to prove \reftext{\eqref{E:riesz_p1}} for
$q=\infty $.
If $p = \frac{n}{\gamma}$, $q=1$, and, for instance, $\alpha_0=0$ and $\alpha_\infty>0$, then, by the computations above, $\lVert S_{\frac{n}{\gamma}} g^*\rVert_{p',q';-\mathbb{A}}\approx\lVert g\rVert_{(1,\infty;[0,-\alpha_\infty])}$. Hence \eqref{E:riesz_p1} for this particular case follows from the description of the associate space of $L^{(1,\infty;[0,-\alpha_\infty])}$ provided by \citep[Theorem~6.7]{OP}. The other cases when $p = \frac{n}{\gamma}$ and $q=1$ can be proved similarly.
In the remaining cases the proof is analogous to that of
\reftext{Theorem~\ref{T:fractional-maximal-operator-GLZ}}. We omit the details.
\end{proof}
We finally note that the result stated in
\reftext{Theorem~\ref{T:riesz-potential}} follows in the usual way from the
corresponding result for the weighted Stieltjes transform. Its proof is
analogous to that of \reftext{Theorem~\ref{T:stieltjes-transform}}.
%
\begin{theorem}
Let $\alpha \in (1,\infty )$. Let $X$ be a rearrangement-invariant
Banach function space over $(0,\infty )$ such that
%
\begin{eqnarray}
\label{E:xi-satisfied2}
\xi _{\alpha } \in X'(0,\infty ),
\end{eqnarray}
where $\xi _{\alpha }$ is defined by~\reftext{\eqref{E:definition-of-xi}}. Define
the functional $\sigma $ by
\begin{eqnarray*}
\sigma (f)=\left \|  S_{\alpha }f^{*}\right \|  _{X'(0,\infty )}, \quad f
\in \mathcal{M}_{+}(0,\infty ).
\end{eqnarray*}
Then $\sigma $ is an~r.i.~norm and
%
\begin{eqnarray}
\label{E:boundedness-weighted-stieltjes}
S_{\alpha }\colon X\to Y,
\end{eqnarray}
where $Y=Y(\sigma ')$. Moreover, $Y$ is the optimal
\textup{(}smallest\textup{)} r.i.~space for
which~\reftext{\eqref{E:boundedness-weighted-stieltjes}} holds.

Conversely, if~\reftext{\eqref{E:xi-satisfied2}} is not true, then there does not
exist an~r.i.~space $Y$ for
which~\reftext{\eqref{E:boundedness-weighted-stieltjes}} holds.
\end{theorem}

\paragraph{Acknowledgment}
We wish to thank the referee for valuable
comments. We are greatly indebted to Lenka Slav\'{\i }kov\'{a} for
stimulating discussions about the subject.

\paragraph{Funding}
This research was supported by the grants P201-13-14743S and
P201-18-00580S of the Czech Science Foundation, by the grant 8X17028 of the
Czech Ministry of Education and by the grant SVV-2017-260455.

\bibliography{classical-operators-arxiv}

\end{document}